\documentclass[12pt]{imsart}

\usepackage{iftex}

\ifPDFTeX 
  \usepackage[utf8]{inputenc}
  \usepackage[T1]{fontenc}
  \usepackage{lmodern}
\fi

\usepackage[margin=1in]{geometry}
\usepackage{graphicx,xcolor,mathrsfs,bbm}

\usepackage[toc]{appendix}
\RequirePackage[numbers]{natbib}
\usepackage{amsmath,amssymb,amsthm,bm,latexsym,mathtools}
\usepackage{amsfonts}
\usepackage{dsfont}
\usepackage{orcidlink}
\usepackage{comment}
\usepackage{enumitem}
\usepackage{url}
\usepackage{booktabs}
\usepackage{hyperref}
\hypersetup{colorlinks=true,citecolor=blue,linktocpage=true}
\usepackage{cleveref}

\startlocaldefs

\makeatletter
\newcommand{\namedlabel}[2]{\begingroup #2\def\@currentlabel{\textnormal{#2}}%
  \phantomsection\label{#1}\endgroup}
\makeatother
  
\setlist[enumerate,1]{label=(\roman*)}
\setlist[enumerate,2]{label=(\alph*)}

\def\E{\mathbb{E}}
\def\P{\mathbb{P}}

\def\R{\mathbb{R}}
\def\C{\mathbb{C}}

\newcommand\EE{\mathbb{E}}
\newcommand\Indic[1]{\Ind_{\{#1\}}}
\newcommand\Ind{\mathbf{1}}
\newcommand\kf{\mathtt{k}}

\DeclareMathOperator*{\arginf}{arg\,inf}
\DeclareMathOperator{\supp}{supp}

\AtBeginDocument{\let\Re\tRe}

\newtheorem{dummy}{***}[section]
\newtheorem{theorem}[dummy]{Theorem}

\newtheorem{definition}[dummy]{Definition}

\newtheorem{lemma}[dummy]{Lemma}

\newtheorem{proposition}[dummy]{Proposition}

\AddToHook{env/theorem/begin}{\crefalias{dummy}{theorem}}
\AddToHook{env/example/begin}{\crefalias{dummy}{example}}
\AddToHook{env/corollary/begin}{\crefalias{dummy}{corollary}}
\AddToHook{env/definition/begin}{\crefalias{dummy}{definition}}
\AddToHook{env/assumption/begin}{\crefalias{dummy}{assumption}}
\AddToHook{env/lemma/begin}{\crefalias{dummy}{lemma}}
\AddToHook{env/remark/begin}{\crefalias{dummy}{remark}}
\AddToHook{env/proposition/begin}{\crefalias{dummy}{proposition}}
\crefname{theorem}{Theorem}{Theorems}
\crefname{example}{Example}{Examples}
\crefname{corollary}{Corollary}{Corollaries}
\crefname{definition}{Definition}{Definitions}
\crefname{assumption}{Assumption}{Assumptions}
\crefname{lemma}{Lemma}{Lemmas}
\crefname{remark}{Remark}{Remarks}
\crefname{proposition}{Proposition}{Propositions}

\numberwithin{equation}{section}  

\endlocaldefs

\begin{document}

\begin{frontmatter}
\title{Long-term behaviour of refracted Lévy processes in a half-line}
\makeatletter\runtitle{\@title}\makeatother

\author[UCL]{\fnms{Ramkrishna Jyoti} \snm{Samanta}\ead[label=e1]{ramkrishna.samanta.24@ucl.ac.uk}}
\author[UCL]{\fnms{Alexander R.} \snm{Watson}\ead[label=e2]{alex.watson@ucl.ac.uk}}
\runauthor{Samanta and Watson}

\address[UCL]{Department of Statistical Science, University College London, 
Gower Street, London WC1E 6BT, UK}
\begin{abstract}
  We study the long-term behaviour of refracted Lévy processes within a half-line,
  with killing on exiting the domain and at state-dependent rate inside it.
  We identify the decay rate of the killed semigroup together with the
  associated invariant function and measure, and we obtain convergence at
  exponential rate to a Yaglom limit.
  The proofs make use of $R$-theory and Lyapunov function techniques.
\end{abstract}
\end{frontmatter}

% \tableofcontents
\section{Introduction}\label{sec:introduction}

Consider a Markov process $V$ on state space $E\cup\{\partial\}$, in which
$\partial$ is an absorbing state. Denote by $\tau_\partial$ the time at which
$V$ reaches $\partial$. Fundamental to the study of $V$ are its
\emph{quasi-limiting distributions}, those probability measures $\nu$ on $E$
such that
\[
  \P_{\mu}(V_t \in \cdot \mid \tau_\partial > t)
  \to \nu,
\]
for some initial distribution $\mu$ of $V_0$,
and with a suitable notion of convergence.
When this convergence holds for all Dirac measures $\delta_x$, $x\in E$,
we say that $\nu$ is a \emph{Yaglom limit}.
Any quasi-limiting distribution is also \emph{quasi-stationary}, which means
that
\[
  \P_{\nu}(V_t \in \cdot \mid \tau_\partial > t)
  = \nu.
\]
Conversely, a quasi-stationary distribution $\nu$ is always quasi-limiting
when the initial distribution is $\nu$, but may not be for other distributions,
and in particular need not be a Yaglom limit.

If we study the sub-probability measures $\P_\mu(\cdot; \tau_\partial>t)$
we may see $V$ as a process in $E$ killed at time $\tau_\partial$; from
our point of view, absorption and killing are equivalent.

The study of quasi-limiting distributions, including their existence, 
uniqueness and the rate at which convergence occurs, has a long history.
The textbook of \citet{ColletMartinezSanMartin2013} gives a comprehensive
overview, including applications to birth-death processes and absorbed
diffusions. The work of \citet{ChampagnatVillemonais2023GeneralCriteria}
provides useful Lyapunov-type criteria for convergence, and there is also a
strong connection with $R$-theory \cite{TuominenTweedie1979}.

Perhaps surprisingly, quasi-limiting distributions are useful
not only in absorption situations, but also for branching processes, which model
a collection of individuals, each with a trait evolving in time,
which give birth (or branch) at random, resulting in a growing
population.
Roughly speaking, a non-conservative semigroup, which captures
the number of individuals of a certain trait at each time, can
be linked to a Markov process with killing. In turn, its quasi-limiting
behaviour describes the long-term behaviour
of the branching process, including the growth in population size.
This connection is made in the context
of growth-fragmentations by \citet{VillemonaisWatson2025GrowthFragmentationQSD};
see also \citet{Cavalli2020} for a model closely related to the present work.

In this work, we are interested in the quasi-limiting behaviour of a
refracted Lévy process killed on exiting a half-line $(b,\infty)$.
Let us begin by defining the object of study.

A \emph{Lévy process} is a stochastic process $X$ with stationary, independent increments
and càdlàg paths. $X$ is called \emph{spectrally negative} if it has no positive jumps
and does not have almost surely increasing paths. The distribution of $X$ can be
characterised using its \emph{Laplace exponent} $\psi_X\colon\R \to \R\cup\{\infty\}$,
\[
  \EE_0[e^{\lambda X_t}] = e^{t\psi_X(\lambda)},
  \qquad
  \lambda \in \R,
\]
where $\P_x$ and $\EE_x$ represent probabilities and expectations with the
process having initial value $x\in\R$.

The function $\psi_X$ is convex, and is finite at least on $[0,\infty)$.
Where it is finite, it is smooth and has the \emph{Lévy-Khintchine}
representation
\[
  \psi_X(\lambda) = \gamma_X \lambda
  + \frac{\sigma_X^2}{2}\lambda^2
  + \int_{(-\infty,0)} \bigl( e^{\lambda x} - 1 - \lambda x \Indic{x>-1} \bigr) \Pi_X(d x).
\]
Here, $\gamma_X\in \R$ is the \emph{centre}, $\sigma_X\ge 0$ is the \emph{Gaussian coefficient}
and $\Pi_X$, a measure on $(-\infty,0)$, is called the \emph{Lévy measure}.
We will have reason to make use of the inverse of $\psi_X$, the function
$\Phi_X \colon\R \to \R\cup\{-\infty\}$ given by
\[
  \Phi_X(q) = \sup\{ \lambda \in \R : \psi_X(\lambda) = q \},
  \qquad q\in \R,
\]
which is concave.
For further details on Lévy processes, we refer to
\cite{Ber-Levy,Kyprianou2014}.

The class of \emph{refracted Lévy processes} was described by \citet{KyprianouLoeffen2010}
and \citet{NobaYano2016} (see below for remarks on terminology).
Let $X$ and $Y$ be two spectrally negative Lévy processes, and
assume that $X$ has
zero Gaussian coefficient. The
\emph{refracted Lévy process based on $(X,Y)$}
is a Feller process $U$ which behaves like $X$ when above zero and
like $Y$ when below zero. More precisely, when $U$ is started above
zero, it behaves like $X$ until passing into the negative half-line, and then
behaves like $Y$, and when started below zero, it behaves like $Y$
until passing into the positive half-line. The need for a rigorous
construction becomes obvious when considering how the process should
behave when started from the point zero itself, and \citet{NobaYano2016}
provide this using Itô synthesis.

Historically, refracted Lévy processes were introduced by
\citet{KyprianouLoeffen2010}, who considered the case where
$Y_t = X_t - \delta t$ for some $\delta >0$,
and \citet{NobaYano2016}
subsequently called their process the `generalised refracted Lévy process'. For brevity,
and since we do not need to distinguish these two classes, 
we drop the word `generalised'.
Refracted Lévy processes arise naturally in actuarial science,
in the context of
optimal control problems with
absolutely continuous controls \cite{KLP-opt-control,NPY-refraction}
and `optimal new business' \cite{LoeffenMartinezVanSchaik2020}.

To the refracted process $U$, we introduce additional state-dependent killing
(sometimes called `omega-killing' in the literature).
First fix two killing rates $q_X, q_Y\ge 0$, and define
\[
  \kf(x) = q_X\Indic{x\ge 0} + q_Y\Indic{x<0},
  \qquad x\in\R.
\]
A killing time is introduced as
\[
  \zeta = \inf\biggl\{ t > 0: \int_0^t \kf(U_s)\, ds > \mathbf{e} \biggr\},
\]
where $\mathbf{e}\sim \text{Exp}(1)$ is independent of $U$. Enlarging the state
space with a cemetery state $\partial$, we define the stochastic process
$V = (V_t)_{t\ge 0}$ by
\[
  V_t = U_t\Indic{t<\zeta} + \partial \Indic{t\ge \zeta}.
\]
For bounded measurable $f$, adopting the usual convention that $f(\partial)=0$,
we have
\[
  \EE_x[f(V_t)] = \EE_x\Bigl[e^{-\int_0^t \kf(U_s) ds } f(U_t)\Bigr].
\]
The process $V$ is our main object of study.
It could be understood as taking $X$ killed at rate $q_X$ and $Y$ killed at rate $q_Y$,
and constructing the refracted Lévy process corresponding to these. However, killed
Lévy processes are not considered by \citet{NobaYano2016}, so we find it preferable
to introduce this killing by a multiplicative functional.

We are now in a position the state our goal. Let $b<0$, and define the
first passage time
\[
  \tau_b^- = \inf\{ t\ge 0: V_t < b\}.
\]
Let \(Q = (Q_t)_{t\ge 0}\) be the sub-Markovian semigroup
associated with $V$ killed upon first passage below $b$:
\begin{equation}\label{eq:def_Q_t}
Q_t f(x)
=\E_x\bigl[ f(V_t);\, t<\tau_b^- \wedge \zeta \bigr],
\qquad x>b,    
\end{equation}
for bounded measurable functions \(f\).
We are interested in the asymptotic properties of $Q$,
most especially the quasi-limiting distribution of the process.
Note that the choice $b<0$ is not a restriction, since otherwise we reduce
the question to one about the Lévy process $X$; we will discuss this in more
detail after stating our result.

We work under the following assumptions.
\begin{enumerate}
    \item[\namedlabel{a:W}{\textnormal{(W)}}]
    The process $X$ has no Gaussian component.

    \item[\namedlabel{a:R}{\textnormal{(R)}}]
    The process $X$ drifts to $-\infty$, or equivalently, $\psi'_X(0+)<0$.

    \item[\namedlabel{a:A1}{\textnormal{(A1)}}]
    At least one of the following holds:
    \begin{enumerate}
        \item[\namedlabel{a:A1a}{\textnormal{(a)}}]
        there exists $\ell<b$ belonging to the support of $\Pi_X$;

        \item[\namedlabel{a:A1b}{\textnormal{(b)}}]
        there exists $\ell\in(b,0)$ belonging to the support of $\Pi_Y$.
    \end{enumerate}

    \item[\namedlabel{a:A2}{\textnormal{(A2)}}]
    The point $0$ belongs to the support of $\Pi_X$.

    \item[\namedlabel{a:A3}{\textnormal{(A3)}}]
    The point $0$ belongs to the support of $\Pi_Y$.

    \item[\namedlabel{a:A4}{\textnormal{(A4)}}]
    The Lévy measure $\Pi_X$ has a non-trivial absolutely continuous
    component with density $\pi_X$, and there exist constants
    $0<\alpha_1<\alpha_2$ such that
    $
        \inf_{(-\alpha_2,-\alpha_1)}\pi_X>0.
    $

    \item[\namedlabel{a:A5}{\textnormal{(A5)}}]
    The Lévy measure $\Pi_Y$ has a non-trivial absolutely continuous
    component with density $\pi_Y$, and there exists $\alpha>0$ such that
    $
        \inf_{(-\alpha,0)}\pi_Y>0.
    $
\end{enumerate}

Assumption~\ref{a:W} is required in \cite{NobaYano2016} for the construction of the refracted
Lévy process, while
Assumption~\ref{a:R} will ensure that $V$ returns to zero from positive
levels.
The remaining assumptions are non-degeneracy conditions on the jump
mechanisms of $X$ and $Y$. Assumption~\ref{a:A1} ensures that the process
can move from the positive region into the interval $(b,0)$ without being
killed and is used to establish Lebesgue irreducibility.
Assumptions~\ref{a:A2} and~\ref{a:A3} ensure the availability of
arbitrarily small negative jumps for $X$ and $Y$, respectively.
Together with 
\ref{a:A4} and~\ref{a:A5}, they yield simultaneous Lebesgue
irreducibility and the minorization estimates required for the
$R$-theoretic argument.

Notice that Assumption~\ref{a:A5} implies
Assumption~\ref{a:A3} and also \ref{a:A1}\ref{a:A1b}. We nevertheless
retain Assumption~\ref{a:A3} separately, since several intermediate
results require only the weaker support condition and do not use the
density assumption~\ref{a:A5}. 

With the aim now of establishing some notation to be able to give our main
result explicitly, we turn to the two-sided exit problem for $V$. It turns
out that, as with Lévy processes, this admits a simple solution.
For $a>0$, define
\[
  \tau_a^+ = \inf\{ t\ge 0: V_t > a \}.
\]
For each $q\ge 0$, there exists a function $W_V^{(q)} \colon\R \times (-\infty,0)$, called the
\emph{scale function} of $V$, with the property that
\begin{equation}
  \label{eq:intro-exit-V}
  \EE_x \bigl[ e^{-q\tau_a^+} ; \tau_a^+ < \tau_b^- \wedge \zeta \bigr]
  =
  \frac{W_V^{(q)}(x,b)}{W_V^{(q)}(a,b)}.
\end{equation}
This mirrors the situation for a Lévy process, say $X$, where the problem is solved
in terms of a univariate function $W_X^{(q)}$. Indeed, we will later give
a representation of $W_V^{(q)}$ in terms of $W_X^{(q+q_X)}$ and $W_Y^{(q+q_Y)}$.
Equation \eqref{eq:intro-exit-V} is an extension of the formula obtained for $U$
by \cite{NobaYano2016}.

In passing from a two-sided exit problem to a one-sided problem,
the function
\begin{equation*}
    D^{(q)}(y)
    =
    \lim_{a\to\infty}
    \frac{
        W_V^{(q)}(a,y)
    }{
        W_X^{(q+q_X)}(a)
    },
    \qquad q\ge 0, \, y\leq 0,
\end{equation*}
will arise.
The asymptotic properties of $V$ killed on first passage below $b$ are formulated
in terms of $W_V^{(q)}$ and $D^{(q)}$. Thus far, we have considered only non-negative,
real values of $q$. However, when $x$ and $y$ are fixed, there are in fact extensions of
the maps $q\mapsto W_V^{(q)}(x,y)$
and $q\mapsto D^{(q)}(y)$ which are analytic on the whole complex plane.

We define
\begin{equation}\label{eq:sigma-b-introduction}
    \sigma(b)
    =
    \inf
    \left\{
        p\geq 0:
        D^{(-p)}(b)=0
    \right\},
\end{equation}
which will prove to be the decay rate of $Q$.
Before stating our main result, we need a few final pieces of notation.
For a probability measure $\mu$, we write $\P_\mu = \int \P_x \mu(dx)$ for the probability where
$V_0$ is distributed as $\mu$. For any measure $\nu$ and measurable function $f$,
we write $\nu f = \int f(x) \nu(dx)$ for integration.
For a signed measure $\nu$ and a positive function $\psi$ on a space $E$,
we define a weighted total variation norm
\[
  \lVert \nu\rVert_{\mathrm{TV}(\psi)}
  = \sup\{ \lvert \nu f\rvert :  f\colon E \to \R\text{ measurable s.t. } \lvert f\rvert \le \psi\}.
\]
Finally, to fit the framework of quasi-stationarity, we introduce measures with
explicit killing: the process $V$ under $\P^\dag_x$ is exactly the same as it
is under $\P_x$ up to the time $\tau_b^-$, at which point it is sent to the
cemetery state $\partial$.

\begin{theorem}
\label{thm:main-introduction}
Assume that \ref{a:W}, \ref{a:R} and
\textnormal{\ref{a:A1}--\ref{a:A5}} hold, and $\sigma(b) < q_X-\inf \Psi_X$.
Then:
\begin{enumerate}
\item
$\sigma(b)>0$ and the function $W_V^{(-\sigma(b))}(\cdot, b)$
is strictly positive on $(b,\infty)$, and the measure
\begin{equation}\label{eq:nu-introduction}
    \nu(dy)
    =
    \left(
        D^{(-\sigma(b))}(y)
        \mathbf{1}_{(b,0]}(y)
        +
        e^{-\Phi_X(q_X-\sigma(b))y}
        \mathbf{1}_{(0,\infty)}(y)
    \right)
    dy
\end{equation}
is non-zero and has finite total mass.
\item
The function $W_V^{(-\sigma(b))}(\cdot,b)$ and the measure $\nu$ are
$\sigma(b)$-invariant for $Q$, in that
\begin{equation}\label{eq:right-left-invariance-introduction}
    Q_tW_V^{(-\sigma(b))}(\cdot,b)(x)
    =
    e^{-\sigma(b)t}W_V^{(-\sigma(b))}(x,b),
    \text{ and }
    \nu Q_t f = e^{-\sigma(b) t} \nu f,
\end{equation}
for every $x>b$ and bounded measurable function $f$.
\item
There exist functions $\psi_1,\psi_2$ (with the property that
$\lim_{x\to\infty} \psi_1(x) = \infty$) and constants
$C,\gamma >0$ such that, for all probability
measures $\mu$ on $(b,\infty)$ with $\mu \psi_1<\infty$ and $\mu\psi_2>0$,
\begin{equation}\label{eq:main-asymptotic-introduction}
\Big\|
\mathbb{P}^\dag_\mu\!\big( V_t \in \,\cdot \mathbin{\big|} t < \tau_\partial \big)
- \frac{\nu(\cdot)}{\nu(b,\infty)}
\Big\|_{\mathrm{TV}(\psi_1)}
\leq C e^{-\gamma t} \frac{\mu\psi_1}{\mu\psi_2},
\qquad t>0.
\end{equation}
\end{enumerate}
\end{theorem}

In particular, the probability measure
$\frac{\nu(\cdot)}{\nu(b,\infty)}$
is the Yaglom limit corresponding to $Q$,
and the convergence takes place at exponential rate
and in a weighted total variation sense.

\medskip \noindent
Among the first works considering invariant distributions of killed Lévy processes, and the
one to which we owe the most in terms of approach, is that of \citet{Bertoin1997}. There,
Bertoin considers a spectrally negative Lévy process $X$, killed on exiting a compact interval $[0,a]$.
His approach uses $R$-theory. He first finds a $q$-resolvent density for the killed process,
for any $q\ge 0$, in terms of the $W_X^{(q)}$. By extending the density to complex $q$,
he finds it blows up at the first negative zero of the function $\R \ni q\mapsto W^{(q)}(a)$.
By appealing \citet{TuominenTweedie1979} and making use of a Tauberian argument, he deduces
an invariant function and measure for the killed process, corresponding to (i--ii) of
\cref{thm:main-introduction}.
Building on this work, \citet{Pis-exit} obtained analogous results for a reflected Lévy process.

Turning to the question of a Lévy process killed upon exiting a half-line, we find that
\citet{KyprianouPalmowski2006QuasiStationaryLevy} used Wiener--Hopf methods to characterise
the Yaglom limit in terms of its exponential moments. This work covers two wide classes of
Lévy processes, including ones with two-sided jumps and stable-like behaviour.
A detailed study of the spectrally positive case was made by \citet{Yamato2023_spectrally_positive_Levy},
including a full classification of the quasi-stationary distributions, again given
in terms of the region of positivity of the scale function.

When looking at a spectrally positive Lévy process killed on exiting $(b,\infty)$, the
process is continuous at the killing time. This property was fruitfully exploited in
later works: \citet{noba2024analyticpropertygeneralizedscale}
defined scale functions for downwards skip-free standard Markov processes,
and Yamato subsequently characterised the quasi-stationary distributions
for a process on the natural numbers \cite{yamato2023qsd}
and a real interval \cite{Yamato2026EntranceBoundary}, as well as
obtaining a Yaglom limit (at exponential rate)
in the case of `coming down from infinity'.
The quasi-stationary distributions are given in terms of the generalised
scale functions.
Other work in this vein, with representations in terms of the
resolvent under the excursion measure, can be found in \cite{Cerf-qsd,GT-taboo-regen}.

A final recent work of interest is that of \citet{Cavalli2020}, which is a study
of a growth-fragmentation process in which the growth rate of cells depends on
whether their size is above or below a certain threshold. Describing the
asymptotic mean behaviour
of this branching process is essentially equivalent to looking for the Yaglom limit
of a refracted Lévy process (in Kyprianou and Loeffen's sense) with
killing at state-dependent rate but without any killing at first passage.
Cavalli takes a slightly different approach, following \citet{BertoinWatson2018}.

To summarise our position within the literature, we are considering a problem
where the process is discontinuous at the first passage killing, which
places us in a different position from \cite{Yamato2023_spectrally_positive_Levy,yamato2023qsd,Yamato2026EntranceBoundary}
and other works in that vein. We follow the $R$-theory programme of
\cite{Bertoin1997,Pis-exit}, though there is an interesting departure in
terms of our assumptions. In order to meet certain irreducibility hypotheses,
it is assumed in \cite{Bertoin1997,Pis-exit} that the one-dimensional distributions of the Lévy process were
absolutely continuous; this is fairly strong and, in particular, excludes compound
Poisson processes with drift.
In this article we have been at pains to make weaker hypotheses, which led us to
assumptions (A1--5).
Though some improvement is no doubt possible, we think that these will cover many situations
of interest.

Finally, like \citet{yamato2023qsd,Yamato2026EntranceBoundary} but unlike
most of these works, we study the rate of convergence.
Our approach uses the general criteria for
exponential convergence to quasi-stationarity developed by
\citet{ChampagnatVillemonais2023GeneralCriteria}.
Since our state space $(b,\infty)$ is unbounded, it is also interesting
to investigate the dependence on the initial condition, since one
cannot expect uniformity.

\medskip\noindent
The remainder of this work is organised as follows.
In \cref{resolvent-densities-analytic}, we find expressions for the
resolvent of the refracted Lévy process killed outside an interval
and outside a half-line in terms of the functions $W_V^{(q)}$ and $D^{(q)}$,
and we study their analytic properties.
In \cref{s:irreducibility-properties}, we use (A1--5) to obtain the
Lebesgue irreducibility (\cref{lebesgue_irr})
and simultaneous Lebesgue irreducibility (\cref{sim-leb-irr}) of the refracted
Lévy process killed outside of a half-line, which are critical
for applying $R$-theory, as well as minorisation results which are
needed for the quasi-limiting convergence.
The $R$-theory project is carried out in \cref{s:positive-recurrence}, in which
we establish $\sigma(b)$-positive recurrence of the
process (\cref{lem:positive-recurrence}) and find
its invariant function and measure, with
\cref{thm:main-introduction}(i--ii) appearing as \cref{lem:rho_invariant_qsd}. 
Finally, in \cref{s:long-term},
we obtain convergence at exponential rate, providing
\cref{thm:main-introduction}(iii) as \cref{lem:combined-conditions}.

\section{Resolvent densities and analytic properties}\label{resolvent-densities-analytic}

We begin our study of the killed refracted Lévy process by establishing some
identities. Our aim is the resolvent of the process killed on exiting the
half-line $(b,\infty)$. This is conveniently expressed in terms of scale
functions, and we start out by recapping the situation for Lévy processes.
In the second half of the section, we take a look at the analytic properties
of these functions; we remind the reader that our aim in $R$-theory is to
identify a simple pole of the resolvent density.

Consider now the Lévy process $X$, and define its first-passage times
\begin{equation*}
    T_a^+
    =
    \inf\{t>0:X_t>a\}
    \text{ and }
    T_c^-
    =
    \inf\{t>0:X_t<c\}.
\end{equation*}
It is well-known \cite[section~8]{Kyprianou2014} that for any $q\ge 0$
and $c<x<a$,
\[
  \E_x\bigl[ e^{-q T_a^+} ; T_a^+ < T_c^- \bigr]
  = \frac{W_X^{(q)}(x-c)}{W_X^{(q)}(a-c)}.
\]
This identity is known as the \emph{two-sided exit problem}, and the function
$W^{(q)}_X$ is the \emph{scale function} of $X$. It satisfies $W_X^{(q)}(x) = 0$
for $x<0$ and on the positive half-line it admits a representation
in terms of Laplace transforms:
\[
  \int_0^\infty e^{-\theta x} W_X^{(q)}(x) \, dx
  =
  \frac{1}{\psi_X(\theta)-q},
  \qquad
  \theta > \Phi_X(q).
\]
$W_X$ is non-negative,
continuous and increasing on $\mathbb R$, and is strictly positive
on $(0,\infty)$.
We next define an analogous function for $V$, before proceeding to study resolvent
densities.

For $q\geq 0$, define the scale function \(W_V^{(q)} \colon \R \times (-\infty,0) \to \R \) associated with  \(V\) by
\[
W_V^{(q)}(x,y) :=
\begin{cases}
\displaystyle
\frac{q+q_X}{\Phi_X(q+q_X)}\,
W_X^{(q+q_X)}(x)\,W_Y^{(q+q_Y)}(-y)
\\[2ex]
\displaystyle\quad
{} +\iint
\Big(
e^{-\Phi_X(q+q_X)v}\,W_X^{(q+q_X)}(x)\,W_Y^{(q+q_Y)}(-y)
\\
\qquad\qquad\quad {} - W_X^{(q+q_X)}(x-v)\,W_Y^{(q+q_Y)}(-y+u)
\Big)\,
\tilde{\Pi}_X(du,dv),
& x>0, \\
W_Y^{(q+q_Y)}(x-y),
& x\le 0,
\end{cases}
\]
where
\[
\tilde{\Pi}_X(du, dv)
=
\mathbf{1}_{\{u<0,\, v>0\}}\, \Pi_X(du-v)\, dv,
\]
and $\Pi_X$ denotes the Lévy measure of $X$. 

Using standard properties of $W_X^{(q+q_X)}$ and $W_Y^{(q+q_Y)}$, one can see that
for fixed $y<0$, $W_V^{(q)}(\cdot, y)$ is non-negative and increasing and is
continuous on $\R\setminus\{0\}$, and that $W_V^{(q)}(x,y)>0$ when $y<x$.

Recall that $\tau_a^+$ and $\tau_b^-$ are first passage times for $V$, and that,
unless otherwise stated, we always have $f(\partial) = 0$.
The following pair of lemmas is obtained in the special case $q_X=q_Y=0$
(no killing) by \citet{NobaYano2016}.
\begin{lemma}[two-sided exit problem for $V$]\label{two-sided-exit}
For $q\geq 0$ and $b<0<a$,
\[
\mathbb{E}_x\!\left(
e^{-q\tau_a^+};\,
\tau_a^+ < \tau_b^-\wedge \zeta
\right)
=
\frac{W_V^{(q)}(x,b)}{W_V^{(q)}(a,b)}.
\]
\end{lemma}

\begin{lemma}[resolvent killed outside of an interval]
\label{exist_resolvent}
For \(q\geq0\), \(b<0<a\), and any non-negative measurable function
\(f\),
\begin{equation}
\label{gen_resolvent}
\mathbb{E}_x\!\left[
\int_0^{\tau_a^+ \wedge \tau_b^-}
e^{-qt} f(V_t)\,dt
\right]
=
\int_{(b,a)} f(y)\, r_{(b,a)}^{(q)}(x,y)\, dy,
\end{equation}
and the $q$-resolvent density of $V$ killed outside the interval $[b,a]$ is
\[
r_{(b, a)}^{(q)}(x, y) =
\begin{cases}
\dfrac{W_V^{(q)}(x, b)}{W_V^{(q)}(a, b)}
\,W_X^{(q+q_X)}(a-y)
- W_X^{(q+q_X)}(x-y),
& y \in (0,a], \\[1.2em]
\dfrac{W_V^{(q)}(x, b)}{W_V^{(q)}(a, b)}
\,W_V^{(q)}(a,y)
- W_V^{(q)}(x,y),
& y \in [b,0).
\end{cases}
\]
\end{lemma}
The extension from $U$ to $V$ is quite direct, so we omit the proofs.

\begin{lemma}\label{lem:limit_potential}
For \(q\geq 0\), the \(q\)-resolvent density of the process \(V\)
killed upon first passage below \(b<0\) exists and is given by
\[
r^{(q)}_b(x, y) =
\begin{cases}
\displaystyle 
\frac{e^{-\Phi_X(q+q_X)y}}{D^{(q)}(b)}W_{V}^{(q)}(x,b) \,  - W^{(q+q_X)}_X(x-y), 
& y \in (0,\infty), \\[2ex]
\displaystyle 
\frac{D^{(q)}(y)}{D^{(q)}(b)} \, W_{V}^{(q)}(x,b) - W_{V}^{(q)}(x,y), 
& y \in (b,0],
\end{cases}
\]
where $D^{(q)} \colon (-\infty,0) \to \R$ is given by
\[
D^{(q)}(y) = \frac{q+q_X}{\Phi_X(q+q_X)}\,W_Y^{(q+q_Y)}(-y)
+\iint \Big( W^{(q+q_Y)}_Y(-y)\, - W^{(q+q_Y)}_Y(u-y)\Big)\,e^{-\Phi_X(q+q_X)v} \; \tilde{\Pi}_X(du , dv).
\]
\end{lemma}

\begin{proof}
To begin with, we see from \cite[equation~(8.11)]{Kyprianou2014} that,
for $a,v\ge 0$ and for $q\in\R$ such that $q+q_X\ge 0$,
\begin{equation}
  \label{eq:exit-problem-monotonicity}
  \frac{W_X^{(q+q_X)}(a-v)}{W_X^{(q+q_X)}(a)}
  =
  \E\bigl[ e^{-(q+q_X)\tau_v^+} ; \tau_v^+ < \tau_{v-a}^-\bigr]
  \le
  \E\bigl[ e^{-(q+q_X)\tau_v^+} ; \tau_v^+ < \tau_{v-a'}^-\bigr],
  \qquad \text{for } a\le a',
\end{equation}
and so, using \cite[equation~(3.15)]{Kyprianou2014},
\begin{equation}
  \label{eq:WX-limit}
  \frac{W_X^{(q+q_X)}(a-v)}{W_X^{(q+q_X)}(a)}
  \mathbin{\nearrow}
  \E\bigl[ e^{-(q+q_X)\tau_v^+} ; \tau_v^+<\infty \bigr]
  = e^{-\Phi_X(q+q_X)v},
  \qquad a\to\infty.
\end{equation}
We will show that, for all $y\le 0$,
\begin{equation}
  \label{eq:WV-limit}
  \lim_{a\to\infty} \frac{W_V^{(q)}(a,y)}{W_X^{(q+q_X)}(a)}
  = D^{(q)}(y) \in (0,\infty),
\end{equation}
and from these two statements, the claim about resolvents follows by a simple calculation.

Starting with the left-hand side of \eqref{eq:WV-limit},
\begin{align}\label{dini-type}
  \frac{W_V^{(q)}(a,y)}{W_X^{(q+q_X)}(a)}
  &=
  \frac{q+q_X}{\Phi_X(q+q_X)} W_Y^{(q+q_Y)}(-y)\nonumber \\
  & \quad {}
  + \iint \biggl( W_Y^{(q+q_Y)}(-y) \biggr. \nonumber\\
  & \qquad\qquad\quad 
  \biggl. {} - e^{\Phi_X(q+q_X)v}\frac{W_X^{(q+q_X)}(a-v)}{W_X^{(q+q_X)}(a)}
  W_Y^{(q+q_Y)}(u-y) \biggr)
  e^{-\Phi_X(q+q_X)v} \tilde{\Pi}_X(du, dv).
\end{align}
By \eqref{eq:WX-limit}, the integrand is a decreasing function of $a$,
and thus by dominated convergence, the limit as $a\to\infty$ is given by
$D^{(q)}(y)<\infty$.
Moreover, $W_Y^{(q+q_Y)}(-y)-W_Y^{(q+q_Y)}(u-y) > 0$ for $u \in (-y,0)$,
so $D^{(q)}(y)>0$. This establishes \eqref{eq:WV-limit}.

Now, considering the resolvent density $r^{(q)}_{(b,a)}$ when $y>0$, we have
\[
  r_{(b,a)}^{(q)}(x,y)
  = \frac{W_X^{(q+q_X)}(a-y)}{W_X^{(q+q_X)}(a)}
  \frac{W_X^{(q+q_X)}(a)}{W_V^{(q)}(a,b)}
  W_V^{(q)}(x,b)
  - W_X^{(q+q_X)}(x-y).
\]
The considerations above show that, for fixed $y$, the first and second ratio
are both increasing in $a$ and converge to our putative expression for
$r^{(q)}_b(x,y)$. Applying monotone convergence to the integral
\eqref{gen_resolvent} restricted to $(0,\infty)$ completes the proof for $y>0$.

Take now $y\le 0$. In this case we have
\[
  r_{(b,a)}^{(q)}(x,y)
  = \frac{W_V^{(q)}(a,y)}{W_X^{(q+q_X)}(a)}
  \frac{W_X^{(q+q_X)}(a)}{W_V^{(q)}(a,b)}
  W_V^{(q)}(x,b)
  - W_V^{(q)}(x,y).
\]
Now, we see that for fixed $y$, the first ratio is decreasing in $a$ to $D^{(q)}(y)$,
and the second is increasing to $1/D^{(q)}(b)$.
Take any $a_0>0$. This function of $y$ is dominated (for any $x$ and uniformly in $a\ge a_0$) by
$g(y) \coloneqq \frac{W_V^{(q)}(a_0,y)}{W_X^{(q+q_X)}(a_0)} \frac{1}{D^{(q)}(b)} W_V^{(q)}(x,b) - W_V^{(q)}(x,y)$.
Since $W_V^{(q)}$ is bounded on compacts, $\int_b^0 g(y)\, dy < \infty$, and therefore we can apply 
dominated convergence in the integral \eqref{gen_resolvent} over the domain $[b,0]$.
This completes the proof for $y\le 0$.
\end{proof}

\begin{lemma}\label{lem:weighted-to-unweighted}
  For any $q\ge 0$, $x \in \R$, and $\epsilon > 0$,
% originally the discounting was e^{-\Phi_X(\inf \Psi_X)v} and this was called $R^{(q)}$
  \begin{align}\label{A-express}
    0 \le A(q, \epsilon, x) \coloneqq \iint \Big(W_Y^{(q)}(-x)-W_Y^{(q)}(u-x)\Big)e^{-\epsilon v}\,\tilde\Pi_X(du,dv)
    < \infty.  
  \end{align}
  and is integrable in $(b, 0).$
\end{lemma}

\begin{proof}
We know, since it appears in the expression for $D^{(q)}(x)$, \eqref{eq:WV-limit} implies that
\begin{equation} \label{eq:Dq_finite}
  \iint \left(W_Y^{(q)}(-x)-W_Y^{(q)}(u-x)\right)e^{-\Phi_X(q)v}\,\tilde{\Pi}_X(du, dv)<\infty.
\end{equation}
Define
\[
f_{q,\epsilon}(u, v) \coloneqq e^{-\epsilon v}\left(W_Y^{(q)}(-x)-W_Y^{(q)}(u-x)\right)\ge0,\qquad u\le 0, v\geq 0.
\]
Fix \(M>0\) and split the integral into the regions \(\{v\le M\}\) and \(\{v>M\}\).
For \(0\le v\le M\), we have \(e^{-\Phi_X(q)v}\ge e^{-\Phi_X(q)M}>0\). Therefore,
\[
\iint_{v\le M} f_{q,\epsilon}(u, v)\,\tilde\Pi_X(du,dv)
\le e^{\Phi_X(q)M}\iint_{v\leq M}  \left(W_Y^{(q)}(-x)-W_Y^{(q)}(u-x)\right)e^{-\Phi_X(q)v}\,\tilde\Pi_X(du,dv),
\]
which is finite by \eqref{eq:Dq_finite}.
Next, take \(v>M\).
Since \(W_Y^{(q)}(u-x)\ge 0\) for \(u<0\), we have
\begin{align*}
  \iint_{v> M} f_{q,\epsilon}(u, v)\,\tilde\Pi_X(du,dv)
  &\le W_Y^{(q)}(-x) \iint_{v> M} e^{-\epsilon v}\,\tilde\Pi_X(du,dv) \\
  &= \frac{W_Y^{(q)}(-x)}{\epsilon} \int_{(-\infty,M]} (e^{\epsilon w} - e^{-\epsilon M}) \, \Pi_X(d w)
  < \infty,
\end{align*}
where we performed the change of variables $w = u - v$ and used Tonelli's theorem.
This completes the proof of finiteness. 

Next, we focus on integrability. Indeed, the second term is integrable over the compact interval $[b,0]$ because the scale function is continuous. We now deal with the first term. Note that
\begin{multline*}
    \iint_{v\leq M} \bigl(W_Y^{(q)}(-x)-W_Y^{(q)}(u-x)\bigr) e^{-\Phi_X(q)v}
  \Pi_X(du-v) dv 
  \\
  \le 
  \iint \bigl(W_Y^{(q)}(-x)-W_Y^{(q)}(u-y)\bigr) e^{-\Phi_X(q)v}
  \Pi_X(du-v) dv,
\end{multline*}
which appears in the expression of $D^{(q)}$ in \cref{lem:limit_potential}, and is thereby integrable since $q$-resolvent is finite.
\end{proof}

By taking a similar approach as in \cite{noba2024analyticpropertygeneralizedscale},  we prove that $W_{V}^{(q)}(x, y)$ has an analytic extension in $q \in \C$, see \cref{prop:WV-analytic} and we prove similar result for the function $D^{(q)}(x).$
\begin{lemma}\label{lem:Dq_analytic}
Let
\[
\mathcal{D} := \{\, q \in \mathbb{C} : \Re(q+q_X) > \inf \Psi_X \,\}.
\]
Then for each fixed $x\le 0$, the map $q \mapsto D^{(q)}(x)$ has an analytic extension in~$\mathcal{D}$.
\end{lemma}

\begin{proof}
Write $D^{(q)}(x)=D^{(q)}_1(x)+D^{(q)}_2(x)$ where
\[
D^{(q)}_1(x) := \frac{q+q_X}{\Phi_X(q+q_X)}\,W_Y^{(q+q_Y)}(-x)
\]
and
\[
D^{(q)}_2(x) :=
\iint 
\Big( W^{(q+q_Y)}_Y(-x) - W^{(q+q_Y)}_Y(u-x) \Big)
e^{-\Phi_X(q+q_X)v}\, \tilde{\Pi}_X(du, dv).
\]
The function $\Phi$ is the Laplace exponent of a subordinator, which is analytic where finite;
and $q \in \mathcal{D}$ is equivalent
to $\Phi_X(q+q_X)$ being finite and positive.
Moreover, $q\mapsto W_Y^{(q+q_X)}(z)$ has an analytic extension
for each $z$ by Lemma $4$-(ii) in \cite{Bertoin1997}.
This shows that $q \mapsto D^{(q)}_1(x)$ has an analytic extension on~$\mathcal{D}$.

We turn our attention to $D^{(\cdot)}_2(x)$.
Let $\gamma$ be a closed piecewise smooth curve contained in $\mathcal{D}$.  
We will show that $\int_\gamma D^{(q)}_2(x)\, dq =0$.
Consider
\begin{equation} \label{eq:D2-integral}
\int_\gamma D^{(q)}_2(x)\, dq
=
\int_\gamma 
\iint 
\Big( W_Y^{(q+q_Y)}(-x)-W_Y^{(q+q_Y)}(u-x)\Big)
e^{-\Phi_X(q+q_X)v}
\tilde{\Pi}_X(du, dv)\, dq.
\end{equation}
We first control the integrand uniformly in $q$.
Using the analytic extension of scale functions (see Lemma 8.3 in \cite{Kyprianou2014}), for $y, a \in \R,$
\[
W_Y^{(q+q_Y)}(y)-W_Y^{(q+q_Y)}(y-a)
=
\sum_{k\ge 0}(q+q_Y)^k
\Big(W_Y^{*(k+1)}(y)-W_Y^{*(k+1)}(y-a)\Big).
\]
Let $int(\gamma)$ denote the interior of the bounded connected component formed by the closed curve $\gamma$ and define
\[
Q := \sup\{\, |q| : q\in \mathrm{int}(\gamma)\,\}.
\]
Then (using the fact that the convolution of non-negative increasing functions remains increasing),
\begin{align}
\big|W_Y^{(q+q_Y)}(y) - W_Y^{(q+q_Y)}(y-a)\big|
&\le
\sum_{k\ge 0} (Q+q_Y)^k
\Big(W_Y^{*(k+1)}(y)-W_Y^{*(k+1)}(y-a)\Big)
\nonumber \\
&=
W_Y^{(Q+q_Y)}(y)-W_Y^{(Q+q_Y)}(y-a).
\label{eq:mono-diff}
\end{align}
By \cref{lem:weighted-to-unweighted},
\[
\iint 
\Big(
W_Y^{(Q+q_Y)}(-x) - W_Y^{(Q+q_Y)}(-x+u)
\Big)e^{-\Phi_X(\inf \Psi_X)v}
\tilde{\Pi}_X(du, dv)
<\infty.
\]
Thus, recalling that $\gamma$ is contained in $\mathcal{D}$,
the integrand in \eqref{eq:D2-integral} is bounded in absolute value by an integrable function
uniformly in $q \in \mathrm{int}(\gamma)$. It also follows from dominated convergence
that $q\mapsto D_2^{(q)}(x)$ is continuous on $\mathcal D$.
Fubini's theorem implies that
\[
\int_\gamma D^{(q)}_2(x)\, dq
=
\iint 
\left( \int_\gamma 
\big( W_Y^{(q+q_Y)}(-x)-W_Y^{(q+q_Y)}(u-x)\big)
e^{-\Phi_X(q+q_X)v}dq
\right)
\tilde{\Pi}_X(du, dv).
\]

For each fixed $(u,v)$, the function
\[
q \mapsto 
\big( W_Y^{(q+q_Y)}(-x)-W_Y^{(q+q_Y)}(u-x)\big)
e^{-\Phi_X(q+q_X)v}
\]
is analytic on $\mathcal{D}$; hence its integral around $\gamma$ is zero.  
Therefore,
$\int_\gamma D^{(q)}_2(x)\, dq=0$,
and by Morera’s theorem, $D^{(\cdot)}_2(x)$ has analytic extension in~$\mathcal{D}$. Hence, proved.
\end{proof}

\begin{lemma}\label{lem:Dq_bound}
Fix $\inf \Psi_X-q_X<-Q < 0$ and recall the function $A$ appearing in \cref{lem:weighted-to-unweighted}. Then,
\[
|D^{(q)}(x)| \leq g(Q)W_Y^{(Q+q_Y)}(-x)+ A(Q+q_Y,\Phi_X(\inf\Psi_X), x),
\qquad -Q\le q<0, \, x\le 0,
\]
where $g$ is a positive function, and the right-hand side is integrable on $(b,0)$.
Moreover,
\[
  |W_{V}^{(q)}(x, y)| \leq W_{V}^{(Q)}(x, y),
  \qquad
  -Q\le q<0, \, x\in \R , \, y < 0.
\]
and the right-hand side is integrable in $y$ on $(b,0)$.
\end{lemma}

\begin{proof}
We use the same decomposition in \cref{lem:Dq_analytic}, namely $D^{(q)}(x)=D^{(q)}_1(x)+D^{(q)}_2(x)$, and begin with $D_1^{(q)}(x)$.

\citet[Lemma~4(iii)]{Bertoin1997} gives the bound
$|W_Y^{(q+q_Y)}(-x)|\leq W^{(Q+q_Y)}_Y(-x)$.
By Assumption \ref{a:R}, $\Phi_X$ is positive and continuous on $[-Q,\infty)$,
so $g(Q) \coloneqq \sup\bigl\{ \frac{q+q_X}{\Phi_X(q+q_X)} : -Q \le q < 0 \bigr\}$
is finite. This completes the bound on $D_1^{(q)}(x)$.

We turn to $D_2^{(q)}(x)$.
Again using the comparison in \eqref{eq:mono-diff} and monotonicity of $\Phi_X$, we obtain
\begin{align*}
    |D^{(q)}_2(x)| \leq
&\int 
\Big| W^{(q+q_Y)}_Y(-x) - W^{(q+q_Y)}_Y(u-x) \Big|e^{-\Phi_X(q+q_X)v}
\, \tilde{\Pi}_X(du\, dv)\\
& \leq \int 
\Big(
W_Y^{(Q+q_Y)}(-x) - W_Y^{(Q+q_Y)}(-x+u)
\Big)e^{-\Phi_X(\inf \Psi_X)v}
\tilde{\Pi}_X(du\, dv)
\\&= A(Q+q_Y, \Phi_X(\inf\Psi_X), x),
\end{align*}
which is finite by \cref{lem:weighted-to-unweighted}. This completes the bound on
$D_2^{(q)}(x)$.

Finally we consider $W_V^{(q)}(x, y)$. This follows from the analytic extension in \cref{prop:WV-analytic}, that is
\begin{align}
    |W_V^{(q)}(x, b)| \leq  \sum_{n\geq 0} |q|^n W^{*(n+1)}_V(x, b).
\end{align}
Considering the range of $q$ produces the required bound.
\end{proof}

\section{Irreducibility properties}
\label{s:irreducibility-properties}

Our goal in this section and the next is to apply  the $R$-theory developed in \citet{TuominenTweedie1979} in order to
characterize the quasi-limiting behavior of the process $V$ killed upon first passage below $b$. To this end, we
establish \emph{Lebesgue irreducibility} and \emph{simultaneous Lebesgue
irreducibility} of the associated transition semigroup. The definitions of these properties are recalled from from \cite{TuominenTweedie1979}.

\begin{definition}[(simultaneous) $\varphi$-irreducibility]
  Let $(S,\mathcal{S})$ be a measurable space, $\varphi$ be a non-trivial,
$\sigma$-finite measure on $(S,\mathcal{S})$ and 
$P = (P_t)_{t\ge 0}$ a transition semigroup on $(S,\mathcal{S})$.
\begin{enumerate}
\item $P$ is called \emph{$\varphi$-irreducible} if,
for every $x \in S$ and every measurable set $A \in \mathcal{S}$ satisfying
$\varphi(A) > 0$,
\[
\int_0^\infty P_t(x,A)\,dt > 0.
\]
\item 
$P$ is called \emph{simultaneously $\varphi$-irreducible} if, for
every $h>0$, the associated $h$-skeleton chain $(P_{nh})_{n\ge 0}$ is
$\varphi$-irreducible; that is, for every $x\in S$ and every
$A\in\mathcal{S}$ with $\varphi(A)>0$,
\[
\sum_{n=1}^{\infty} P_{nh}(x,A) > 0.
\]
\end{enumerate}
\end{definition}
In the following result we write $H_y = \inf\{ t\ge 0: U_t=y\}$
for the first hitting time of $y$ by $V$.

\begin{lemma}\label{irreducibility}
Under Assumption \ref{a:A1}, for all $x,y>b$,
\[
  \mathbb{P}_x\big( H_y < \tau_b^-\wedge \zeta \big) > 0.
\]
\end{lemma}
\begin{proof}
The proof substantially follows \citet[Proposition 1]{Bertoin1997}, and in part
makes use of that result directly.

To begin with, we note that since $\zeta$ occurs at finite rate, we may retain or
omit it (i.e., consider $V$ or $U$)
from the probability under consideration without changing the question of
positivity, and we so freely in what follows.

Suppose first that $x\le y$. Then
\begin{align}\label{upward_passage_positivity}
  \P_x(H_y<\tau_b^-\wedge \zeta)
  = \frac{W_V^{(0)}(x,b)}{W_V^{(0)}(y,b)} > 0,
\end{align}
since $W_V^{(0)}(\cdot, b)$ is strictly increasing.
For $y<x$, it is sufficient to show that $\P_x(\tau_y^-<\tau_b^-) > 0$.

When $0<y<x$, it is clear that $\P_x(\tau_y^-<\tau_b^-) > 0$, since this reduces
to a question about the Lévy process $X$.

Assume now that, in Assumption (A1), there exists some $b<\ell<0$ in the support
of $\Pi_Y$. When $y<x<0$, \cite[Proposition~1]{Bertoin1997} shows that
\[
  \P_x(H_y<\tau_b^-) \ge \P_x(H_y < \tau_b^-\wedge \tau_0^+) > 0,
\]
since the latter probability concerns just the Lévy process $Y$.
On the other hand when $y<0<x$, we note that by considering just $X$,
$\P_x(\tau_0^- < \tau_b^-)>0$, and on this event we
apply the Markov property at $\tau_0^-$ and reduce to the preceding case.
This concludes the proof under one of the alternatives of Assumption (A1).

Now assume the other alternative of Assumption (A1):
there exists some $\ell<b$ in the support of
$\Pi_X$, and fix $y<x$ with $y<0$, which is the case that remains open.
$\Pi_X(\ell-\varepsilon,\ell+\varepsilon)>0$
for some $0<\varepsilon<y-b$.
Notice that $y-\ell-\varepsilon/2>0$.
Since we have already proved the result when starting from and hitting
points above $0$, we know that
$\P_x(H_{y-\ell-\varepsilon/2} < \tau_b^-)>0$.
On this event, using the a.s.\ right-continuity of paths at \(H_{y - \ell - \varepsilon/2}\), 
for some \(\delta > 0\) we have with positive probability that
\[
  U_s \in (y - \ell - \varepsilon,\, y - \ell), \qquad 
  s \in [H_{y - \ell - \varepsilon/2},\, H_{y - \ell - \varepsilon/2} + \delta].
\]
Within this short interval of time, the process may, again with positive probability,
experience a jump
with size in $(\ell-\varepsilon,\ell+\varepsilon)$ with positive probability.
On this event, \(U\) lands below level \(y\), but above \(b\), before exiting the interval \((b, \infty)\). This completes the proof.
\end{proof}

\begin{proposition}\label{lebesgue_irr}
Under Assumption \ref{a:A1}, the process $V$ killed upon first passage below $b$ is Lebesgue irreducible.
\end{proposition}
\begin{proof}
It follows from \cref{irreducibility}
that for sufficiently large $L$ and for $x,y\in (b,L)$,
\begin{equation}\label{eq:reachability}
    \mathbb{P}_x\big( H_y < \tau_L^+ \wedge \tau_b^-\wedge \zeta \big) > 0.
\end{equation}
The case, when $x<y$, it is clear that the resolvent density in \cref{exist_resolvent} is positive, since $W_V(x, y), W^{(q_X)}(x-y)$ both are identically zero.

So, we are concerned only with the case when $x>y.$ For the case, when $y<0,$
using this type of event a few times (and noticing that one can hit $y<x$ only
by passing below it first) we obtain
\begin{align}
    0 
    < \P_x\big( \tau_b^-\wedge \zeta >\tau_L^+  > \tau_y^-  \big)
    &= \P_x(\tau_b^-\wedge \zeta> \tau_L^+) - \P_x(\tau_y^- \wedge \zeta> \tau_L^+)
    \nonumber \\
    &= \frac{W_{V}(x, b)}{W_{V}(L, b)} - \frac{W_{V}(x, y)}{W_{V}(L, y)},
    \label{eq:WV-comparison}
\end{align}
with the last line following from the two-sided exit problem, \cref{two-sided-exit}. 

Looking at the resolvent density $r_{(b,L)}^{(0)}$ in \cref{exist_resolvent},
we see that $r_{(b,L)}^{(0)}(x,y) > 0$ for all $b<y<0$.
Similarly, for the case when $y>0,$ we notice that 
\begin{align*}
    \P_x(\tau_y^- \wedge \zeta> \tau_L^+)=\frac{W_X^{(q_X)}(x-y)}{W_X^{(q_X)}(L-y)},
\end{align*}
using Lévy process arguments, and substitute in \eqref{eq:WV-comparison} to obtain
\[
  \frac{W_{V}(x, b)}{W_{V}(L, b)} > \frac{W_X^{(q_X)}(x-y)}{W_X^{(q_X)}(L-y)}.
\]
Comparing again with the resolvent density, we see that $r_{(b,L)}^{(0)}(x,y)>0$
for $x>y>0$.

Consequently, we have $R^{(0)}_{(b,L)}(x,A) > 0$ for the resolvent applied to
any $A$ of positive
Lebesgue measure. Notice that this proves Lebesgue irreducibility of the
process killed on exiting $(b,L)$; this will be used in a subsequent proof.

We note here, for future use, that these arguments are valid
for $L \ge \max\{-\sup(\supp \Pi_X) + b, 0\}$ 
under \ref{a:A1}\ref{a:A1a}, and for
any $L>0$ under \ref{a:A1}\ref{a:A1b}.

Since $L$ can be chosen arbitrarily large, this completes the proof.
\end{proof}

At this stage we are nearly in a position to prove the simultaneous Lebesgue irreducibility of $V$ killed upon first passage. We first require some finer analysis of the underlying Lévy processes. Recall that $T_c^-$ and $T_d^+$ refer to the
first passage times of $X$.

\begin{lemma}\label{lem:positivity_survival_general}
Assume \ref{a:A2}. For any $c<d$, any $x\in (c,d)$ and any time $t>0$,
\[
\mathbb{P}_x\big(T_d^+ \wedge T_c^- > t \big) > 0.
\]
The same result holds with $X$ replaced by $Y$ and \ref{a:A2} replaced by \ref{a:A3}.
\end{lemma}

\begin{proof}
Fix $c<d$ and $x\in (c,d)$.
Let us write $H_x^{(n)}$ for the $n$-th hitting time of $x$
by $X$ before exiting the interval $(c,d)$, with $H^{(0)}_x = 0$
and $\{H_x^{(n)}=\infty\}$ denoting the event that $X$ leaves the interval before the $n$-th return to $x$. Under Assumption \ref{a:A2}, \cite[Proposition~1]{Bertoin1997} implies that there exist $h, \epsilon>0$ such that $\P_x\left(H_x\in(h, \infty)\right)>\epsilon.$
For every $t>0$,
\begin{align*}
    \P_x(\tau_c^-\wedge \tau_d^+ > t)
    &\geq \P_x\left(H_x^{(n)}-H_x^{(n-1)}\in(h, \infty), \text{ for all } n\leq [t/h]\right)\\
    &\geq \P_x\left(H_x\in (h, \infty)\right)^{[t/h]}>\epsilon^{[t/h]}>0.
\end{align*}
This concludes the proof.
\end{proof}

\begin{lemma}\label{lem:minorization-X}
Suppose Assumptions \ref{a:A2} and \ref{a:A4} hold.
Let \(A\subset(0,\infty)\) be a set of positive Lebesgue measure.
Then for every starting point \(x\geq0\) there exist a time \(t_0=t_0(x,A)>0\),
a function $m^{x,A}\colon (t_0,\infty) \to (0,\infty)$ and a measure
$\mu_A$ on $(0,\infty)$
such that
\[
\mathbb{P}_x\big( X_t \in dz \; ; t<\tau_0^-\big) \ge m^{x,A}(t) \mu_A(dz),
\qquad t > t_0.
\]
Moreover, $\mu_A(A)>0$ and, for any $t_0<t_1<t_2$, $\inf_{[t_1,t_2]} m^{x,A} > 0$.
\end{lemma}
\begin{proof}
    Let \(A \subset (0, \infty)\) be a set of positive Lebesgue measure. 
Fix a small \(\varepsilon > 0\), and partition \((0, \infty)\) into intervals of length 
strictly less than \(\alpha_2 - \alpha_1\). 
There exists at least one interval in this partition, 
say \([a_1,a_2]\), such that $A \cap [a_1,a_2]$ has positive Lebesgue measure.
By construction, $a_2 - a_1 < \alpha_2 - \alpha_1$,
which implies
\[
0 < a_2 + \alpha_1 < a_1 + \alpha_2,
\]
so the open interval $(a_2 + \alpha_1, \, a_1 + \alpha_2)$
is nonempty. Choose an interval  \( I := (\underline{I}, \overline{I}) \), satisfying
\begin{align}\label{assum-I}
a_2 + \alpha_1 < \underline{I} < \bar{I} < a_1 + \alpha_2
\quad \text{and} \quad 
\bar{I} - \underline{I} < a_1.    
\end{align}
For any $x\ge 0$ and $y \in I$, \citet[Proposition 1]{Bertoin1997} implies that there exists
$t_0 = t_0(x,y,A)$ such that
\begin{equation}\label{eq:minor-1:x-to-y}
  \P_x(H_y < t_0 \wedge \tau_0^-) > 0.
\end{equation}

We decompose the process \(X\) as
\[
X_t = \widetilde{X}_t + J_t,
\]
where $\tilde{X}$ and $J$ are independent, with $\tilde{X}_0=y$ and $J_0=0$;
\(J\) is a compound Poisson process
whose Lévy measure is the restriction of \(\Pi_X\) to
\([-\alpha_2,-\alpha_1]\).

Let \(N_t^J\) denote the number of jumps of \(J\) up to time \(t\).
Let $\delta = \inf_{u \in (-\alpha_2,-\alpha_1)} \pi_X(u)$,
which is positive by Assumption \ref{a:A4}.
For $y\in I$, \(z \in [a_1,a_2]\) and $t>0$, we compute as follows,
using the fact that $\bar{I}-\underline{I}<a_1$ in the first line:
\begin{align}
\mathbb{P}_y(X_t \in dz\; , \tau_0^->t)
&\ge \mathbb{P}_y\!\left( \tilde{X}_t + J_t \in dz \,,\, N_t^J = 1\;, \tilde{\tau}_I>t \right) \nonumber \\
&\ge \int_I \mathbb{P}_y(\tilde{X}_t \in du\; , \tilde{\tau}_I>t)\, \P(N^J_t=1)\frac{\Pi_X(dz - u)}{\Pi_X([-\alpha_2, -\alpha_1])} \nonumber \\
&= \Pi_X([-\alpha_2, -\alpha_1])te^{-\Pi_X[-\alpha_2, -\alpha_1]t}\int_I \mathbb{P}_y(\tilde{X}_t \in du, \tilde{\tau}_I>t)\, \frac{\Pi_X(dz - u)}{\Pi_X([-\alpha_2, -\alpha_1])}
\nonumber \\
&\ge \delta te^{-\Pi_X[-\alpha_2, -\alpha_1]t}\mathbb{P}_y(\tilde{\tau}_I>t)\, dz.
\label{dubrosin-condition-X-1}
\end{align}
The last line uses \(z - u \in [\alpha_1, \alpha_2]\), which follows from the
construction of $I$.

The Lévy measure of $\tilde{X}$,
$\Pi_X(\cdot \setminus [-\alpha_2, -\alpha_1])$, still satisfies Assumption \ref{a:A2}.
Therefore, by \cref{lem:positivity_survival_general}, the probability that
\(\tilde{X}\) remains within \(I\) up to \(t>0\) can be chosen arbitrarily
large, is also strictly positive, that is,
\begin{equation}\label{eq:tilde-survival-positive}
  \mathbb{P}_y(\tilde{\tau}_I > t) > 0, \qquad y\in I.  
\end{equation}

We now consider starting at $x\ge 0$ not necessarily in $I$.
By the strong Markov property at $H_y$, for every $t>t_0$,
\begin{align*}
\P_x(X_t\in dz;\,t<\tau_0^-)
&\geq
\int_{[0,t_0)}
\P_x(H_y\in ds,\,H_y<\tau_0^-)
\P_y(X_{t-s}\in dz;\,t-s<\tau_0^-).
\end{align*}
Combining this with \eqref{dubrosin-condition-X-1}, we obtain
\begin{equation}\label{dubrosin-condition-X}
    \P_x(X_t\in dz;\,t<\tau_0^-)
    \geq m^{x,A}(t)\mathbf{1}_{A\cap[a_1,a_2]}(z)\,dz
\end{equation}
where
\begin{equation*}
m^{x,A}(t)
=
\delta\int_{[0,t_0)}
(t-s)e^{-\Pi_X[-\alpha_2,-\alpha_1](t-s)}
\P_y(\widetilde\tau_I>t-s)
\P_x(H_y\in ds,\,H_y<\tau_0^-).
\end{equation*}
By \eqref{eq:minor-1:x-to-y} and
\eqref{eq:tilde-survival-positive}, $m^{x,A}(t)>0$ for every $t>t_0$.

Finally, if $t\in[t_1,t_2]$ with $t_1>t_0$, then
$t-s\in[t_1-t_0,t_2]$ for $s\in[0,t_0)$. Therefore,
\begin{align*}
m^{x,A}(t)
&\geq
\delta
\inf_{r\in[t_1-t_0,t_2]}
\bigl(re^{-\lambda r}\bigr)
\P_y(\widetilde\tau_I>t_2)
\P_x(H_y<t_0\wedge\tau_0^-) > 0.
\end{align*}
Since the right-hand side is independent of $t\in[t_1,t_2]$ this concludes the proof.
\end{proof}

\begin{lemma}\label{lem:minorization-Y}
Suppose Assumptions \ref{a:A3} and \ref{a:A5} hold.
Let \(A\subset(b, 0)\) be a set of positive Lebesgue measure.
Then for every starting point \(x\in(b, 0)\) there exists a time \(t_0=t_0(x,A)>0\),
a function $\tilde{m}^{x,A}\colon (t_0,\infty) \to (0,\infty)$ and a measure
$\tilde{\mu}_A$ on $(b, 0)$
such that
\[
\mathbb{P}_x\big( Y_t \in dz \; ; t<\tau_0^+\wedge\tau_b^-\big) \ge \tilde{m}^{x,A}(t) \tilde{\mu}_A(dz),
\qquad t > t_0.
\]
Moreover, $\tilde{\mu}_A(A)>0$ and, for any $t_0<t_1<t_2$, $\inf_{[t_1,t_2]} \tilde{m}^{x,A} > 0$.
\end{lemma}
\begin{proof}
Recall that Assumption~\ref{a:A5} implies
Assumption~\ref{a:A1}, so we are free to use results proved under
the latter condition.
Let \(A \subset (b,0)\) be a measurable set of positive Lebesgue measure. By the Lebesgue density theorem, there exists a density point
\[
a_1 \in A, \qquad b<a_1<\sup A \le 0,
\]
such that $A\cap[a_1,a_2]$ has positive Lebesgue measure for every $a_2>a_1$.
Using \ref{a:A5}, let $0<\gamma_1<\gamma_2$ be such that $\delta\coloneqq \inf_{[-\gamma_2,-\gamma_1]}\pi_Y>0$
and $a_1+\gamma_2<0$.
Now take any $a_2$ with the property that
\[
  a_1 < a_2 < a_1+\gamma_2-\gamma_1.
\]
With these conditions in place, choose an interval $I = (\underline{I},\bar{I}) \subset (b,0)$
such that $\bar{I}-\underline{I}<a_1-b$ and 
\[
  a_2+\gamma_1 \le \underline{I} < \bar{I} \le a_1+\gamma_2.
\]
Having established the existence of intervals $I$ and $[a_1,a_2]$ with these properties,
the proof proceeds almost identically to that of \cref{lem:minorization-X}
(albeit using \cref{irreducibility} to go from starting points in $I$ to ones in $(b,0)$)
and therefore we omit it.
We note that the difference between the proofs lies in the fact that $a_1$ and $\gamma_2$
depend on the set $A$, which explains the fact that \ref{a:A5} is stronger than
the corresponding assumption on $X$. \qedhere
\end{proof}

\begin{proposition}\label{sim-leb-irr}
    Under Assumptions (A2--5), the process $V$ killed upon going below $b$ is simultaneously Lebesgue irreducible.
\end{proposition}
    \begin{proof}
        Initially, fix a set $A\subset (0, \infty)$ of positive Lebesgue measure.
        Start by considering $x>0$. 
        From \cref{lem:minorization-X}, there exists $t_0 = t_0(x,A)$ such that
        for $t>t_0$,
       \begin{align*}
           \P_x(V_t\in A; \;t<\tau_0^-)>0.
       \end{align*}
       (The presence of killing at time $\zeta$ does not affect this result.)
       When $x \in (b,0]$, applying \cref{irreducibility} allows us to reduce to the
       case of starting above zero. Again similar argument holds for $A \subset (b, 0)$ by \cref{lem:minorization-Y}.
       
       This will imply that for $x\in (b, 0),$ there exists a time \(t_0=t_0(x,A)>0\) such that
for all \(t>t_0\)
\[
\mathbb{P}_x\big( V_t \in A \; ; t<\tau_0^+\wedge \tau_b^-\big) > 0.
\] 
Finally, take $x \in [0, \infty),$ we again use \eqref{upward_passage_positivity} and iteratively, we have the following. That is, for all $x \in (b, \infty)$ and any set $A\subset(b, \infty)$ of Lebesgue positive measure ($\{0\}$ is ignored as it is Lebesgue zero set), there exits $t(x, A)$ such that for all $t>t(x, A):$
\[
\mathbb{P}_x\big( V_t \in A \; ; t<\tau_b^-\big) > 0.
\] 
Thus, we have that for all $h>0$, $x\in (b, \infty)$ and positive Lebesgue measure set $A\subset (b, \infty)$, the $0$-resolvents are positive.
    \end{proof}

We establish a strong Doeblin-type minorization condition for the process. This condition allows us to extend results initially proved for particular sets to a sufficiently rich class of Borel sets. It will be used to verify positive recurrence and, subsequently, to derive exponential convergence results.

\begin{lemma}\label{unifrom-minorizing-measure}
Under the Assumptions \ref{a:A1}--\ref{a:A5}, for a bounded Borel set $L\subset (b, \infty)$ with $\inf L > b$
and for any set $A\subset(b, \infty)$ of positive Lebesgue measure, there exist $t_1=t_1(L, A)$,
a function $m^{L,A}\colon (t_1,\infty) \to (0,\infty)$ and a measure $\mu'_A$ on $(b,\infty)$
such that for all $t>t_1$, $y\in A$,
\begin{align*}
    \inf_{x\in L}\P_x(V_t\in dy; \; t<\tau_b^-)\geq m^{L, A}(t)\mu_A'(dy).
\end{align*}
Moreover $\inf_{[t_2, t_3]}m^{L,A}>0$ for $t_1 < t_2 < t_3$ and $\mu'_A(A) > 0$.
\end{lemma}
\begin{proof}
We consider separately the cases
$A\subset (0,\infty)$ and $A\subset (b,0)$.
We first consider $A\subset (0,\infty)$ and assume that
\[
L=[d_1,d_2]\subset (b,\infty), \qquad d_1>b \text{ and } d_2>0,
\]
which can be done by possibly enlarging $L$, without loss of generality.
Next, since $W_V(x,b)>0$ for all $x\in (b,\infty)$, it follows from
\cref{two-sided-exit} that there exists $t_L>0$
such that
\[
\P_{d_1}\!\left(
\tau_{d_2}^+<\tau_{b}^-\wedge t_L\wedge \zeta
\right)>0.
\]
For fixed $x\in L$,
applying the strong Markov property at $\tau_x^+$,
\[
\P_{d_1}\!\left(
\tau_{d_2}^+<\tau_{b}^-\wedge t_L\wedge \zeta
\right)
\le
\P_{d_1}\!\left(
\tau_x^+<\tau_{b}^-\wedge t_L\wedge \zeta
\right)
\,
\P_x\!\left(
\tau_{d_2}^+<\tau_{b}^-\wedge t_L\wedge \zeta
\right),
\]
which yields
\begin{equation}\label{inf-lower-bound-clean}
\inf_{x\in L}\P_x\!\left(
\tau_{d_2}^+<\tau_{b}^-\wedge t_L\wedge \zeta
\right)
\ge
\P_{d_1}\!\left(
\tau_{d_2}^+<\tau_{b}^-\wedge t_L\wedge \zeta
\right)
>0,
\qquad x\in L.
\end{equation}
Applying \cref{lem:minorization-X} and additionally accounting for the killing at time $\zeta$,
there exists $t_0=t_0(d_2,A)>0$ such that for all $t>t_0$,
\[
\P_{d_2}(V_t\in dy,\ t<\tau_0^-\wedge \zeta)
\ge
e^{-q_X t} m^{d_2, A}(t)\mu_A(dy),
\qquad y\in A,
\]
for some $m^{d_2, A}(t)>0$.

Let $t_1 = t_L+t_0(d_2,A)$ and take $t> t_1$ and $x\in L$.
Applying the strong Markov property at $\tau_{d_2}^+$,
we obtain
\begin{align*}
&\P_x(V_t \in dy ,\, t<\tau_b^- \wedge \zeta)
\\&\ge 
\P_x\!\left(
\tau_{d_2}^+<t_L\wedge\tau_b^- \wedge \zeta,\;
t<\tau_b^-\wedge \zeta,\;
V_t\in dy
\right)
\\
&=
\E_x\!\left[
\mathbf{1}_{\{\tau_{d_2}^+<t_L\wedge\tau_b^- \wedge \zeta\}}
\bigl. \P_{d_2}\!\left(
V_{t-s}\in dy,\;
t-s<\tau_b^- \wedge \zeta
\right)\bigr\rvert_{s=\tau_{d_2}^+}
\right]
\\
&\ge
\E_x\!\left[
e^{-q(t-\tau_{d_2}^+)}
\mathbf{1}_{\{\tau_{d_2}^+<t_L\wedge\tau_b^- \wedge \zeta\}}
m^{d_2,A}(t-\tau_{d_2}^+)\right]
\mu^A(dy)
\\
&\ge
e^{-q_Xt}\P_x\!\left(
\tau_{d_2}^+<t_L\wedge\tau_b^- \wedge \zeta
\right)
\inf_{0\le \gamma\le t_L+t_0(d_2, A)}
m^{d_2, A}(t-\gamma)\,\mu_A(dy).
\end{align*}

Combining this with \eqref{inf-lower-bound-clean}, we deduce that
\[
\inf_{x\in L}\P_x(V_t\in dy;\ t<\tau_b^-\wedge \zeta)
\ge m^{L, A}(t)\,\mu_A(dy),
\]
for some $m^{L, A}(t)>0$. Moreover, the positivity
$\inf_{[t_2,t_3]} m^{L,A} > 0$ for $t_1<t_2<t_3$
is inherited from the corresponding property of $m^{d_2, A}$.

The case $A\subset (b,0)$ is treated in an analogous manner. The only difference is that we now take $L=[d_1,d_2]$ with $d_1<0$.
In this case, we use \cref{irreducibility} to replace \eqref{inf-lower-bound-clean} by the inequality
\[
  \inf_{x\in L} \P_x(H_{d_1}<\tau^-_b \wedge t_L \wedge \zeta)
  \ge \P_{d_1}(\tau^+_{d_2} \vee H_{d_1} < \tau^-_b \wedge t_L \wedge\zeta) > 0,
\]
and then proceed as in the proof above, using $d_1$ as a reference point instead of $d_2$,
and \cref{lem:minorization-Y} in place of \cref{lem:minorization-X}.
\end{proof}

We are now in a position to define two important critical rates.
For any $a>0$, we define
\begin{align*}
  \sigma(b,a) &= \inf\{\,p \ge 0 : W_{V}^{(-p)}(a,b) = 0\,\}, \text{ and} \\
  \sigma(b) &= \inf\{\,p \ge 0 : D^{(-p)}(b) = 0\,\}.
\end{align*}

Our approach in the next section will be to use $R$-theory to establish the long-term behaviour of $V$ using its resolvent density. This is inspired by \cite{Bertoin1997}, who studied a totally asymmetric Lévy process $X$ killed upon exiting a compact interval. The definition of $\sigma(b,a)$ is analogous to the rate defined in that work. In the following section, we will show that, for $V$ killed on exiting $(b,\infty)$, the corresponding rate is $\sigma(b)$, and we will also consider the convergence of $\sigma(b,a)$ as $a\to\infty$.

\begin{proposition}\label{cor_1}
\begin{enumerate}
    \item The resolvent density $r^{(q)}_b$ obtained in \cref{lem:limit_potential}
remains valid for $q > \max\{\, \inf \Psi_X-q_X,\, -\sigma(b)\}$.
  \item The resolvent density $r^{(q)}_{(b,a)}$ found in \cref{exist_resolvent} remains valid for $q>-\sigma(b, a)$.
  \item\label{i:cor_1:3}
  If \ref{a:A1}\ref{a:A1a} holds, let $L \ge \max\{-\sup(\supp \Pi_X) + b, 0\}$,
  and if \ref{a:A1}\ref{a:A1b} holds, let $L>0$.
  Then, 
  $\sigma(b, x) \ge \sigma(b,a)$ for all $a\ge L$ and $x \le a$.
\end{enumerate}
\end{proposition}

\begin{proof}
\begin{enumerate}
\item
Start by taking a non-negative, bounded measurable function $f$, supported on $(0, \infty)$.
We know from
\cref{lem:limit_potential} 
that for $q\geq 0,$ for a fixed $x \in \R$ we have:
\begin{equation}\label{local-resolvent-eqn}
\mathbb{E}_x\!\left[
\int_0^{\tau_b^-}
e^{-qt} f(V_t)\,dt
\right]
=
\int_{(b,\infty)} f(y)\, r_{b}^{(q)}(x,y)\, dy.
\end{equation}
From \cref{lem:Dq_analytic} and \cref{prop:WV-analytic},
it can be seen that $r_b^{(q)}(x,y)$ is analytic for $q > \max\{\, \inf \Psi_X-q_X,\, -\sigma(b)\}$. Now, fix $\max\{\, \inf \Psi_X-q_X,\, -\sigma(b) \,\}<q_0<0$.
By assumption (R), we know that there exists $\epsilon = \arginf\psi_X >0$ such that
$e^{-\Phi_X(q+q_X)y}\leq e^{-\epsilon y}$ for $q\ge q_0$ and $y\ge 0$.
Also, note that $C\coloneqq\inf_{0\le q\le q_0}D^{(q)}(b)>0,$ since $D^{(\cdot)}(b)$ is analytic
by \cref{lem:Dq_analytic}. Therefore, for fixed $x$ and $y>0$, we have an analytic extension of the above as follows:
\begin{align*}
    r^{(q)}_b(x, y) = \sum_{n\geq 0} q^n d_n(x, y)
\end{align*}
Therefore, using \cref{lem:Dq_bound}
and the expression for $r_b^{(q)}$ in \cref{lem:limit_potential}, we have 
\begin{align*}
    &\lvert r_b^{(q)}(x,y) \rvert
= \left|\frac{e^{-\Phi_X(q+q_X)y}}{D^{(q)}(b)}W_{V}^{(q)}(x,b) \,  - W^{(q+q_X)}_X(x-y)\right|\\
&
\leq \frac{1}{C}e^{-\epsilon y}W_V^{(q_0)}(x, b)+ W^{(q_0+q_X)}_X(x-y),
\end{align*}
for $q\ge q_0$ and $y\ge 0$.
Since the right-hand side is non-negative and integrable in $y$, we can use Fubini's theorem, to conclude that the integral on the right-hand-side of \eqref{local-resolvent-eqn} also has an extension to $q>q_0,$ and therefore for all $q > \max\{\, \inf \Psi_X-q_X,\, -\sigma(b)\}$. The rest of the argument follows exactly as in Proposition $3$ in \cite{Bertoin1997}
and this concludes the proof for $y > 0$.

For $y<0$, the same argument works with a slight modification to ensure integrability,
as follows.
Assume $q_0\le q\le 0$ where $-\max\{\psi_X-q_X,-\sigma(b)\} < q_0 < 0$ and fix $y\in (b,0)$.
We have by \cref{lem:Dq_bound} that $D^{(q)}\rvert_{(0,b)}$
is bounded (uniformly in such $q$) by an integrable function.
% \[ 
%   \lvert D_2^{(q)}(y) \rvert
%   \le A(q_1,\epsilon, y)
%   = \iint \bigl(W_Y^{(q_1)}(-y) - W^{(q_1)}(u-y)\bigr) e^{-\epsilon v} \Pi_X(du-v)dv,
% \]
% with $q_1 > \max\{ q_Y, \lvert q_Y+q_0\rvert \}$, which by \cref{lem:weighted-to-unweighted} is integrable. Once this is established, we can conclude similarly to the $y>0$ case.

% \textcolor{blue}{We omit the red paragraph. Please read rge above explanation, I think it is useful to use Lemma \ref{lem:weighted-to-unweighted}.}
% {\color{red}Looking at the proof of \cref{cor_1}-(i), 
% we have, for any $M$, the bound
% \begin{align}
%   A(q_1,\epsilon,y)
%   &\le
%   e^{\Phi_X(q_1)M} \iint_{v\leq M} \bigl(W_Y^{(q_1)}(-y)-W_Y^{(q_1)}(u-y)\bigr) e^{-\Phi_X(q_1)v}
%   \Pi_X(du-v) dv \nonumber \\
%   & \quad {} + W_Y^{(q_1)}(-y)
%   \frac{1}{\epsilon} \int_{(-\infty,M]} (e^{\epsilon w} - e^{-\epsilon M}) \Pi_X(dw).
%   \label{eq:cor_1:inter1}
% \end{align}
% The second term is integrable over the compact interval $[b,0]$ because the scale function is continuous. We now deal with the first term. Let $q_2=q_1-q_Y>0$,
% and bound
% \begin{multline*}
%     \iint_{v\leq M} \bigl(W_Y^{(q_1)}(-y)-W_Y^{(q_1)}(u-y)\bigr) e^{-\Phi_X(q_1)v}
%   \Pi_X(du-v) dv 
%   \\
%   \le 
%   \iint \bigl(W_Y^{(q_2+q_Y)}(-y)-W_Y^{(q_2+q_Y)}(u-y)\bigr) e^{-\Phi_X(q_2+q_X)v}
%   \Pi_X(du-v) dv,
% \end{multline*}
% which appears in \cref{lem:limit_potential} and is thereby integrable.}

\item\label{i:cor_1:2}
This is similar to the preceding part, and even a little simpler,
since now $r_{(b,a)}$ is written in
terms of $W_V^{(q)}$ and $W_X^{(q)}$ alone.
We use the bounds on $W_V^{(q)}(x,b)$ from \cref{lem:Dq_bound}
along with classical bounds on the scale functions of $X$.
Since $W_{V}^{(q)}(x, b)$ is analytic in $q$ from \cref{prop:WV-analytic}, these show
that the resolvent representation of \cref{exist_resolvent} remains valid for all \(q > -\sigma(b,a)\).

\item
As remarked at the end of \cref{lebesgue_irr}, for such $L$, the
process killed on exiting $(b, L)$ is Lebesgue irreducible;
see the end of the proof of \cref{lebesgue_irr}.
Let $a\ge L$ and suppose, for contradiction, that, for some \(q < \sigma(b,a)\) and \(x_0 \in (b,a)\),  \(W_V^{(-q)}(x_0,b) = 0\).

Then, by part \ref{i:cor_1:2} and \cref{exist_resolvent}, we see that $r^{(-q)}_{(b,a)}(x_0, y) = 0$ for all $y\in (x_0, a),$ since the first term in $r^{(-q)}_{(b,a)}(x_0, y)$ is zero by the assumption, and the second term $W^{(-q+q_X)}(x_0-y) = 0$ since $y>x_0.$
This contradicts the Lebesgue irreducibility of \(V\).
Hence \(W_V^{(-q)}(x,b) > 0\) for all \(x \in (b,a)\), which establishes the claim.
\end{enumerate}
\end{proof}

\section{Positive recurrence, invariant function and invariant measure}
\label{s:positive-recurrence}

In the rest of the work, Assumptions \ref{a:W}, \ref{a:R} and \ref{a:A1}--\ref{a:A5}
are in force, and we assume moreover that
\begin{equation}\label{condition-Phi}
    \sigma(b) < q_X-\inf \Psi_X.
\end{equation}
holds.

As the following result shows, \eqref{condition-Phi} is satisfied under quite natural
conditions on the processes $X$ and $Y$: opposing drifts and 
common existence of an exponential moment.

\begin{proposition}\label{prop:condition-Phi-example}
Let $q_X=q_Y=0$ and let $Y$ drift to $+\infty$.
Assume moreover that both $\psi_X$ and $\psi_Y$ are finite on some
interval $(-\epsilon, 0)$ with $\epsilon>0$.
Then, there exists $b_0<0$ such that, for every $b\le b_0$,
$
  \sigma(b) < -\inf \Psi_X.
$
\end{proposition}
\begin{proof}
Under the assumptions in the statement, there exists some $q$ with the properties that
\[
  \max\{\inf \Psi_X, \inf \Psi_Y\} < q < 0
  \text{ and }
  \psi_X(\Phi_Y(q)) < \infty.
\]
We note that since $X$ drifts to $-\infty$, 
$    \Phi_X(q)>0$,
and since $Y$ drifts to $+\infty$,
$
    \Phi_Y(q)<0
$.
Moreover, by \cite[Theorem~1.5]{Yamato2023_spectrally_positive_Levy},
\begin{equation*}
    -\inf\Psi_Y
    =
    \inf\left\{
        p\geq0:
        W_Y^{(-p)}(z)\leq0
        \text{ for some }z>0
    \right\}.
\end{equation*}
Since $-q<-\inf\Psi_Y$, it follows that
\begin{equation}\label{eq:WYq-positive}
    W_Y^{(q)}(z)>0,
    \qquad z>0.
\end{equation}

By \cref{lem:limit_potential,lem:Dq_analytic}, for every $x<0$,
\begin{align}
\frac{D^{(q)}(x)}
     {W_Y^{(q)}(-x)}
&=
\frac{q}{\Phi_X(q)}
+
\iint
\left(
    1-
    \frac{W_Y^{(q)}(u-x)}
         {W_Y^{(q)}(-x)}
\right)
e^{-\Phi_X(q)v}\,
\widetilde{\Pi}_X(du,dv).
\label{eq:D-ratio-proof}
\end{align}

For fixed $u<0$, the argument used in the proof of
\cref{lem:limit_potential}, with $Y$ in place of $X$, gives
\begin{equation}\label{eq:ratio-monotone-example}
    \frac{W_Y^{(q)}(u-x)}
         {W_Y^{(q)}(-x)}
    \mathbin{\nearrow}
    e^{\Phi_Y(q)u},
    \qquad x\to-\infty.
\end{equation}
(Here, the positivity in \eqref{eq:WYq-positive} ensures that the same
two-sided exit argument remains valid for the negative parameter $q$.)
Since $u<0$ and $\Phi_Y(q)<0$, the integrand in \eqref{eq:D-ratio-proof}
is eventually negative, and it is decreasing as $x\to-\infty$.
Therefore, we can apply
monotone convergence to show that
\begin{align}
&\lim_{x\to-\infty}
\iint
\left(
    1-
    \frac{W_Y^{(q)}(-x+u)}
         {W_Y^{(q)}(-x)}
\right)
e^{-\Phi_X(q)v}\,
\widetilde{\Pi}_X(du,dv)
\nonumber\\
&\qquad=
\iint
\left(
    1-e^{\Phi_Y(q)u}
\right)
e^{-\Phi_X(q)v}\,
\widetilde{\Pi}_X(du,dv)
%\label{eq:D-ratio-MCT}
\nonumber
\\
&\qquad =
\int_{(-\infty,0)}
\int_0^{-z}
\left(
    1-e^{\Phi_Y(q)(z+v)}
\right)
e^{-\Phi_X(q)v}\,
\mathrm{d}v\,\Pi_X(\mathrm{d}z)
%\label{eq:D-change-variables-example}
\nonumber
\\
&\qquad =
\frac{
    \psi_X\bigl(\Phi_Y(q)\bigr)-q
}{
    \Phi_Y(q)-\Phi_X(q)
} - \frac{q}{\Phi_X(q)}.
\label{eq:D-divided-difference-example}
\end{align}
Combining \eqref{eq:D-ratio-proof} and
\eqref{eq:D-divided-difference-example} proves 
\begin{equation}\label{eq:D-ratio-limit}
    \lim_{x\to-\infty}
    \frac{D^{(q)}(x)}
         {W_Y^{(q)}(-x)}
    =
    \frac{
        \psi_X\bigl(\Phi_Y(q)\bigr)-q
    }{
        \Phi_Y(q)-\Phi_X(q)
    }
    <0.
\end{equation}
To show the negativity, observe that
$
    \Phi_Y(q)<0<\Phi_X(q).
$
Moreover, since $X$ drifts to $-\infty$,
$\psi_X'(0+)<0$. By convexity of $\psi_X$,
\begin{equation*}
    \psi_X\bigl(\Phi_Y(q)\bigr)
    \geq
    \Phi_Y(q)\psi_X'(0+)
    >0>q.
\end{equation*}
Thus, the numerator of \eqref{eq:D-ratio-limit} is positive and the
denominator is negative, so the limit is negative.

By \eqref{eq:WYq-positive}, $W_Y^{(q)}(-x)>0$ for every $x<0$.
It follows from \eqref{eq:D-ratio-limit} that there exists $b_0<0$ such
that
\begin{equation}\label{eq:Dq-negative-small-b}
    D^{(q)}(b)<0,
    \qquad b\leq b_0.
\end{equation}
Fix $b\leq b_0$.  We know from \cref{lem:limit_potential} that
$
    D^{(0)}(b)>0.
$
Since $r\mapsto D^{(r)}(b)$ is continuous by \cref{lem:Dq_analytic},
there exists $r_b\in(q,0)$ such that
$
    D^{(r_b)}(b)=0.
$
By definition, $\sigma(b)$ is the smallest positive zero of $r\mapsto D^{(-r)}(b)$,
so
\begin{equation*}
    \sigma(b)
    \leq -r_b
    <-q
    <-\inf\Psi_X,
\end{equation*}
which completes the proof.
\end{proof}

We now recall the definition of $Q_t$ from \eqref{eq:def_Q_t} and the following definitions from \cite{TuominenTweedie1979}.
\begin{definition}[$\lambda_0$-recurrence, $\lambda_0$-positive recurrence and $\lambda_0$-null recurrence]
Let $P = (P_t)_{t\geq0}$ be a sub-Markovian semigroup on
$(S,\mathcal B)$. A number $\lambda_0\in\mathbb R$ is called the
\emph{decay parameter} if there exist a Lebesgue-null set $N$ and a
countable partition $\mathcal X$ of $S$ such that, for every
$x\in N^c$ and every $A\in\mathcal X$,
\[
    \int_0^\infty e^{s t}P_t(x,A)\,\mathrm dt
    \begin{cases}
    <\infty, & s<\lambda_0,\\
    =\infty, & s>\lambda_0.
    \end{cases}
\]
The semigroup $(P_t)$ is said to be \emph{$\lambda_0$-recurrent} if,
in addition,
\[
    \int_0^\infty e^{\lambda_0 t}P_t(x,A)\,\mathrm dt
    =
    \infty
\]
for every $x\in N^c$ and every $A\in\mathcal B^+$, where
$\mathcal B^+$ denotes the collection of Borel sets with positive
Lebesgue measure.

A non-negative
measurable function $f$ and a non-trivial $\sigma$-finite measure $\pi$
are called respectively a \emph{$\lambda_0$-invariant function} and a
\emph{$\lambda_0$-invariant measure} for $P$ if
\[
    f(x)
    =
    e^{\lambda_0 t}P_t f(x),
    \qquad t\geq0,
\]
for Lebesgue-a.e. $x$, and
\[
    \pi(A)
    =
    e^{\lambda_0 t}\pi P_t\Ind_A,
    \qquad t\geq0,\quad A\in\mathcal B.
\]
A $\lambda_0$-recurrent semigroup $P$ is called \emph{$\lambda_0$-positive recurrent} if
$\pi(f) < \infty$,
and \emph{$\lambda_0$-null recurrent} otherwise.
\end{definition}

In the last section,  we established the Lebesgue irreducibility for the semigroup $Q$.
The significance of this arises from Theorem 2 of \cite{TuominenTweedie1979},
which provides a dichotomy: a Lebesgue irreducible semigroup with decay parameter
$\lambda_0$ is either $\lambda_0$-transient or $\lambda_0$-recurrent. 

Our next steps are to establish the recurrence of the semigroup $Q$ and thereby
obtain its invariant measure and, ultimately, its convergence behaviour.

\begin{lemma}\label{lem:rho_recurrent}
The semigroup $Q$ is \(\sigma(b)\)-recurrent. 
\end{lemma}

\begin{proof}
Lemma 4 in \cite{Bertoin1997} states that the mapping
\[
(x,q)\longmapsto W_V^{(q)}(x,b)=W_Y^{(q+q_Y)}(x-b)
\]
is jointly continuous on $\mathbb{C}\times (b,0)$.

Lemma 5 in \cite{Bertoin1997} gives a lower bound for $W_Y$ in terms of its antiderivative $\bar{W}_Y$, and this implies (noting that $\lim_{\eta\downarrow 0}\bar{W}_Y(\eta)=0$) that there exist $\eta_1,\eta_2>0$ such that
\[
W_V^{(-\sigma(b)+\epsilon)}(x,b)
=
W_Y^{(-\sigma(b)+q_Y+\epsilon)}(x-b)
> 0,
\qquad (\epsilon,x)\in (0,\eta_1)\times (b,b+\eta_2).
\]
Using the joint continuity and positivity, together with the convergence $\lim_{q\to\sigma(b)}D^{(-q)}(b)=0$ already established, we obtain
\[
\lim_{q\uparrow \sigma(b)}
\left(
\frac{W_V^{(-q)}(x,b)}{D^{(-q)}(b)} e^{-\Phi_X(-q+q_X)y}
-
W_X^{(-q+q_X)}(x-y)
\right)
= \infty,
\]
for all $x\in (b,b+\eta_2)$ and $y\in (0,\infty)$.

This divergence implies that the $(-\sigma(b))$-resolvent associated with the killed process (see \cref{lem:limit_potential} and \cref{cor_1}) is infinite on every Borel set $B\subseteq (0,\infty)$ with positive Lebesgue measure. This identifies $\sigma(b)$ as the decay parameter of the process.

Finally, $(b,b+\eta_2)$ has positive Lebesgue measure and the process is Lebesgue irreducible. The divergence of the resolvent density rules out transience, and hence the killed process is $\sigma(b)$-recurrent.
\end{proof}

\begin{lemma}\label{strong-stability}
Fix $x\in (b,\infty)$. Suppose that $A$ is a Borel set
with positive Lebesgue measure, such that for all $h>0,$
\[
\lim_{n\to\infty}
e^{\sigma(b)nh}
\mathbb{P}_x\!\left(
V_{nh}\in A,\,
nh<\tau_b^-
\right)
=0.
\]
Then, for every bounded Borel set $L$ of positive Lebesgue measure with $\inf L>b$,
and for every $h'>0$, we also have
\[
\lim_{n\to\infty}
e^{\sigma(b)nh'}
\mathbb{P}_x\!\left(
V_{nh'}\in L,\,
nh'<\tau_b^-
\right)
=0.
\]
\end{lemma}
\begin{proof}
Suppose, for contradiction, that there exists a bounded Borel set $L$ and $h>0$ such that
\[
\limsup_{n\to\infty}
e^{\sigma(b)nh}
\mathbb{P}_x\!\left(
V_{nh}\in L,\,
nh<\tau_b^-
\right)
=c>0.
\]

Let $t_1(L,A)$ be as in \cref{unifrom-minorizing-measure}, and fix $h_1>0$.
Then, for
$t\in t_1(L,A)+(0,h_1)$, the Markov property gives
\begin{align*}
&e^{\sigma(b)(nh+t)}
\mathbb{P}_x\!\left(
V_{nh+t}\in A,\,
nh+t<\tau_b^-
\right)
\\
&\quad {} =
e^{\sigma(b)(nh+t)}
\int_L
\mathbb{P}_x\!\left(
V_{nh}\in dy,\,
nh<\tau_b^-
\right)
\mathbb{P}_y\!\left(
V_t\in A,\,
t<\tau_b^-
\right)
\\
&\quad {} \ge
e^{\sigma(b)nh}
\mathbb{P}_x\!\left(
V_{nh}\in L,\,
nh<\tau_b^-
\right)
\,
e^{\sigma(b)t}
\inf_{y\in L}
\mathbb{P}_y\!\left(
V_t\in A,\,
t<\tau_b^-
\right).
\end{align*}

By \cref{unifrom-minorizing-measure}, there exists a constant $c_1>0$ such that
\[
e^{\sigma(b)t}
\inf_{y\in L}
\mathbb{P}_y\!\left(
V_t\in A,\,
t<\tau_b^-
\right)
\ge c_1,
\qquad
t\in t_1(L,A)+(0,h_1).
\]
Hence,
\[
e^{\sigma(b)(nh+t)}
\mathbb{P}_x\!\left(
V_{nh+t}\in A,\,
nh+t<\tau_b^-
\right)
\ge
c_1
e^{\sigma(b)nh}
\mathbb{P}_x\!\left(
V_{nh}\in L,\,
nh<\tau_b^-
\right).
\]

Now 
\[
\bigcup_{n\ge0}
\left(
nh+t_1(L,A)+(0,h_1)
\right)
\]
is an unbounded open set in $(0,\infty)$. Since
\[
\limsup_{n\to\infty}
e^{\sigma(b)nh}
\mathbb{P}_x\!\left(
V_{nh}\in L,\,
nh<\tau_b^-
\right)
=c>0,
\]
it follows that for every $\varepsilon\in(0,c)$ there exist infinitely many
$n$ such that
\[
e^{\sigma(b)nh}
\mathbb{P}_x\!\left(
V_{nh}\in L,\,
nh<\tau_b^-
\right)
>
c-\varepsilon.
\]
Let us call that sequence to be $\{n_k\}_{k>0}$ and hence for all $k$ and 
$t\in t_1(L,A)+(0,h_1)$,
\[
e^{\sigma(b)(n_k h+t)}
\mathbb{P}_x\!\left(
V_{n_k h+t}\in A,\,
n_k h+t<\tau_b^-
\right)
\ge
c_1(c-\varepsilon)>0.
\]
Therefore, there exists an unbounded open set of times $s$ along which
\[
e^{\sigma(b)s}
\mathbb{P}_x\!\left(
V_s\in A,\,
s<\tau_b^-
\right)
\]
is bounded away from zero. However, by Theorem~1 of \cite{Kingman1963Ergodic}, there exist $h_2>0$ such that the arithmetic progression
$\{nh_2:n\ge0\}$ intersects such an unbounded open set infinitely often.
Hence,
\[
\limsup_{n\to\infty}
e^{\sigma(b)nh_2}
\mathbb{P}_x\!\left(
V_{nh_2}\in A,\,
nh_2<\tau_b^-
\right)
>0.
\]
This contradicts our assumption on the set $A$. Hence, proved.
\end{proof}
We record here a version of Croft's lemma applied to semicontinuous functions;
see the remark after Corollary~2 in \cite{Kingman1963Ergodic}.
Its proof is a straightforward adaptation of that result and so we omit it.
\begin{lemma}[Croft's lemma]
\label{croft-lemma}
Let $f\ge0$ be lower semicontinuous. If
$
f(nh)\to0  \text{ as } n\to\infty,
$
for all $h>0$,
then
\[
f(t)\to0 \text{ as } t\to\infty.
\]
\end{lemma}
% \begin{proof}
% Suppose otherwise. Then
% \[
% \limsup_{t\to\infty}f(t)=c>0.
% \]
% Choose $0<\varepsilon<c$ and set
% \[
% G:=\{t>0:f(t)>c-\varepsilon\}.
% \]
% Since $f$ is lower semicontinuous, $G$ is open; since $\limsup_{t\to\infty}f(t)=c$, it is unbounded. By Theorem~1 of \cite{Kingman1963Ergodic}, there exists $h>0$ such that
% \[
% G\cap\{nh:n\in\mathbb N\}
% \]
% is infinite. Hence $f(nh)>c-\varepsilon$ for infinitely many $n$, contradicting $f(nh)\to0$. Therefore $f(t)\to0$ as $t\to\infty$.
% \end{proof}
Echoing \cite{Bertoin1997}, we approach the question of positive recurrence
using a Tauberian theorem; see 
\cite{KyprianouPalmowski2006QuasiStationaryLevy,Pis-exit} for some other works in this
tradition.

\begin{proposition}
\label{lem:positive-recurrence}
The semigroup $Q$
is $\sigma(b)$-positive recurrent, and the function $D^{(q)}(b)$ has a simple root at $q=-\sigma(b)$.
\end{proposition}

\begin{proof}
By \cref{lem:Dq_analytic}, the function $q\mapsto D^{(q)}(b)$ is analytic in a neighbourhood of $q=-\sigma(b)$. Hence there exist $n\in\mathbb{N}$ (the order of the root) and a constant $C>0$ such that
\[
D^{(-\sigma(b)+\varepsilon)}(b)
\sim
C\varepsilon^n,
\qquad \varepsilon\downarrow0.
\]
Fix $x\in (b, 0)$ and let $B\subseteq (0,\infty)$ be a Borel set with positive Lebesgue measure. As in the proof of \cref{lem:rho_recurrent}, 
there exist constants $\eta_1,\eta_2>0$ such that
\[
W_V^{(-\sigma(b)+\epsilon)}(x,b)
=
W_Y^{(-\sigma(b)+q_Y+\epsilon)}(x-b)
> 0,
\qquad (\epsilon,x)\in (0,\eta_1)\times (b,b+\eta_2).
\]
Using the representation of the resolvent density $r_b^{(q)}$, together with the asymptotic behaviour of $D^{(-\sigma(b)+\varepsilon)}(b)$, we obtain
\[
\int_0^\infty
e^{-\varepsilon t}
e^{\sigma(b)t}
\mathbb{P}_x\!\left(
V_t\in B,\,
t<\tau_b^-
\right)\,dt
\sim
c(B)\varepsilon^{-n},
\qquad \varepsilon\downarrow0,
\]
for some constant $c(B)>0$.

Applying Karamata's Tauberian theorem (see
\cite[Theorem~1.7.1]{BinghamGoldieTeugels1989}), we deduce that
\begin{equation}\label{eq:tauberian-growth}
\int_0^s
e^{\sigma(b)t}
\mathbb{P}_x\!\left(
V_t\in B,\,
t<\tau_b^-
\right)\,dt
\sim
\tilde c(B)s^n,
\qquad s\to\infty,
\end{equation}
where
\[
\tilde c(B)
=
C
W_V^{(-\sigma(b))}(x,b)
\int_{B\cap(0,\infty)}
e^{-\Phi_X(-\sigma(b)+q_X)y}\,dy
>0.
\]

We now show that the killed process is $\sigma(b)$-positive recurrent.

Suppose, for contradiction, that the killed process is $\sigma(b)$-null recurrent. Since the process is simultaneously Lebesgue irreducible (see \cref{sim-leb-irr}), \cite[Theorem~4]{TuominenTweedie1979} implies that for every $h>0$, the $h$-skeleton chain $(Q_{nh})_{n\ge 0}$ is $e^{\sigma(b)h}$-null recurrent, in the discrete sense.
In particular from Theorem $6$ in \cite{tweedie1974}, there exists  $A\subset (b,\infty)$ and a Lebesgue null set $N \subset (b,\infty)$ such that for all $x \in N^c$ and all $h>0$,
\[
\lim_{n \to \infty} e^{\sigma(b) nh}\,
\mathbb{P}_x\!\left(V_{nh} \in A, \, nh < \tau_b^-\right) = 0.
\]
(The fact that $N$ is independent of $h$ follows from Theorem~6 in \cite{tweedie1974}, since the Borel sigma-algebra algebra is countably generated; see also the proof of Theorem~6(ii) in \cite{TuominenTweedie1979}.)
\begin{comment}
Fix $h_1>0$. Then, by \cite[Theorem~6]{NummelinTweedie1978}, there exists a Borel set $A$ with positive Lebesgue measure such that
\[
\lim_{k\to\infty}
e^{\sigma(b)kh_1}
\mathbb{P}_x\!\left(
V_{kh_1}\in A,\,
kh_1<\tau_b^-
\right)
=0,
\]
for all $x\in N_{h_1}^c$.
\end{comment}
By \cref{strong-stability}, this further implies that for \emph{every} bounded Borel set $G$ of positive Lebesgue measure with $\inf G>b$,
\[
\lim_{n\to\infty}
e^{\sigma(b)nh}
\mathbb{P}_x\!\left(
V_{nh}\in G,\,
nh<\tau_b^-
\right)
=0,
\qquad x\in N^c, \, h>0.
\]

Now let $G\subset (0, \infty)$ be a bounded open interval, with $\inf G>b$. We claim that
\[
t\mapsto
e^{\sigma(b)t}
\mathbb{P}_x\!\left(
V_t\in G,\,
t<\tau_b^-
\right)
\]
is lower semicontinuous. Indeed, for any fixed $t>0$, the process $V$ is almost surely continuous at time $t$, since the probability of a jump occurring at a deterministic time is zero; see the remark following Theorem~3.2 in \cite{kyprianou2020scale}. Hence, by the Portmanteau theorem (see Theorem~2.1 in \cite{Billingsley1968}),
\[
\liminf_{s\to t}
\mathbb{P}_x\!\left(
V_s\in G,\,
s<\tau_b^-
\right)
\ge
\mathbb{P}_x\!\left(
V_t\in G,\,
t<\tau_b^-
\right),
\]
which proves the claim.

We may therefore apply the modified Croft lemma in
\cref{croft-lemma} to conclude that
\[
\lim_{t\to\infty}
e^{\sigma(b)t}
\mathbb{P}_x\!\left(
V_t\in G,\,
t<\tau_b^-
\right)
=0.
\]
However, this contradicts \eqref{eq:tauberian-growth}. Hence, the process is $\sigma(b)$-positive recurrent.
Finally, using \eqref{eq:tauberian-growth} and Theorem 5(i) in \cite{TuominenTweedie1979} gives us $n=1$, so $q=-\sigma(b)$ is a simple root.
\end{proof}

Given the history of this problem, beginning with Bertoin's study \cite{Bertoin1997} of Lévy processes
in a compact interval, it is natural to wonder whether the decay parameter
$\sigma(b)$, corresponding to the process in the half-line, can be obtained by
first considering the exit of a compact interval and taking limits.
The following result gives a positive answer, which will also be important
in obtaining an exponential rate of convergence in \cref{s:long-term}.
\begin{proposition}\label{sigma-monotonicity}
    $\sigma(b, x)  >\sigma(b)$ for all $x\in (b,\infty)$,  and
\[
\lim_{x \to \infty} \sigma(b, x) = \sigma(b).
\]
\end{proposition}
\begin{proof}
We have already established, in \cref{cor_1}\ref{i:cor_1:3},
that $\sigma(b, x)$ is decreasing in $x,$ which implies that the limit exists;
our task is to identify it with $\sigma(b)$.
Assume (for contradiction) that there exists some $x \in (b,\infty)$ such that $\sigma(b, x) \leq \sigma(b)$.  

\emph{Case 1:} $\sigma(b, x) < \sigma(b)$.  
Take a Borel set $B \subseteq (\max\{x,0\},\infty)$ of positive Lebesgue measure.
Using the definition of $\sigma(b,x)$ and the expression for $r_b^{(-\sigma(b,x))}$
given in \cref{lem:limit_potential} and \cref{cor_1},
the $(-\sigma(b,x))$-potential assigned to $B$ is zero. This contradicts the Lebesgue
irreducibility of $V$ (\cref{lebesgue_irr}).

\emph{Case 2:} $\sigma(b, x) = \sigma(b)$.  
Since $W^{(q)}_{V}(x,b)$ is analytic in $q \in \mathbb{C}$, we may write
\[
W_{V}^{(-\sigma(b, x) + \epsilon)}(x,b) \;\sim\; C\, \epsilon^n, 
\qquad \text{for some } C>0, \; n \in \mathbb{N}, \quad \epsilon \to 0.
\]
If $n>1$, then $W_V^{(q)}(x,b)$ has a higher order zero than $D^{(q)}(b)$ at $q=-\sigma(b)$
(see \cref{lem:positive-recurrence}),
which implies (see \cref{exist_resolvent}) that the $-\sigma(b)$-resolvent is zero on
any Borel subset of $(x,\infty)$, which contradicts the Lebesgue irreducibility
proved in \cref{lebesgue_irr}.

Therefore $n=1$, that is, both $W^{(q)}(x,b)$ and $D^{(q)}(b)$ have a
simple zero at $q=-\sigma(b) = -\sigma(b,x)$.
This implies that the resolvent density $r_b^{(-\sigma(b))}$ is finite everywhere, since the assumption would imply $\lim_{q\to \sigma(b)}W^{(-q)}(x, b)/D^{(-q)}(b)<\infty$.
In particular, the resolvent is finite on some compact set of positive Lebesgue measure.
But this contradicts the $\sigma(b)$-recurrence that we obtained in \cref{lem:rho_recurrent}.

Thus, in all cases, we must have $\sigma(b, x) > \sigma(b)$ for all $x \in (b,\infty)$.

Next, since $\sigma(b, x)$ is nonnegative and
decreasing in $x$ (for sufficiently large $x$; \cref{cor_1}),
and $\sigma(b, x)>\sigma(b)$, the limit
\[
R \coloneqq \lim_{x \to \infty} \sigma(b, x)
\]
exists, and $R \in [\sigma(b),\infty)$. Suppose for contradiction, that $R > \sigma(b)$.
Then, for every $x \in (b,\infty)$ and $q \in (\sigma(b), R)$, we must have
\[
W_{V}^{(-q)}(x,b) > 0.
\]
Moreover, by Theorem~1.5 in \citet{Yamato2023_spectrally_positive_Levy}, we know that
\[
\inf \Psi_X \;=\; -\inf\bigl\{ p \geq 0 : W_X^{(-p)}(x) \le 0 \text{ for some } x \bigr\}.
\]
Therefore, whenever $p < q_X-\inf \Psi_X$, it follows that
\[
W_X^{(-p+q_X)}(x) > 0, \qquad x \in (0,\infty).
\]
On the other hand, using the fact that $-\sigma(b)$ is a simple zero of the analytic function
$D^{(\cdot)}(b)$, there exists $\delta > 0$ such that
\[
D^{(-p)}(b) < 0, \qquad p \in (\sigma(b), \sigma(b)+\delta).
\]
Since \eqref{eq:exit-problem-monotonicity} and  \eqref{eq:WX-limit}, hold true even in the case when $q<0,$ we still have
\[
  D^{(-p)}(b) = \lim_{x \to \infty} \frac{W_{V}^{(-p)}(x,b)}{W_X^{(-p+q_X)}(x)}.
\]
Taking $\sigma(b) < p < \max\{ R, q_X-\inf\Psi_X, \sigma(b)+\delta\}$
yields a contradiction: the numerator and denominator on the
right-hand side are positive, while the left-hand side is negative.
We conclude that, indeed, $\lim_{x \to \infty} \sigma(b, x) = \sigma(b)$.
\end{proof}

\begin{theorem}\label{lem:rho_invariant_qsd}
Define a measure \(\nu\) on \((b,\infty)\) by
\[
\nu(dy)
=
\bigl( D^{(-\sigma(b))}(y)\Ind_{(b,0]}(y)
+ e^{-\Phi_X(-\sigma(b)+q_X)\,y}\Ind_{(0,\infty)}(y)\bigr) dy.
\]
\begin{enumerate}[label=(\roman*)]
  \item $W_{V}^{(-\sigma(b))}(\cdot,b)$ is strictly positive on $(b,\infty)$ and is a $\sigma(b)$-invariant function for $Q$, in the sense that
  \[
  Q_t W_{V}^{(-\sigma(b))}(\cdot,b)(x) = e^{-\sigma(b) t}\, W_{V}^{(-\sigma(b))}(x,b),
  \]
  for Lebesgue-almost every $x\in (b, \infty)$.
  
  \item The measure $\nu$ is $\sigma(b)$-invariant for $Q$:
  \[
  \int_{(b,\infty)} Q_t f(y)\,\nu(dy) = e^{-\sigma(b) t}\int_{(b,\infty)} f(y)\,\nu(dy),
  \qquad t\ge 0,
  \]
  for every bounded measurable $f$.
  \item 
  Up to a multiplicative constant, this is the only pair of $\sigma(b)$-invariant function and measure
for $Q$.
\end{enumerate}
\end{theorem}

\begin{proof}
\begin{itemize}
    \item[(i)]
We appeal to \cref{two-sided-exit}. Fix \(x\in(b,\infty)\) and choose
\(a>x\). Then, using markov property, we have
\begin{align*}
    &\E_x\left[
        e^{\sigma(b)t}
        \frac{W_V^{(-\sigma(b))}(V_t,b)}
             {W_V^{(-\sigma(b))}(a,b)};
        \,
        t<\tau_a^+\wedge\tau_b^-\wedge\zeta
    \right] \\
    &\qquad\leq
    \E_x\left[
        e^{\sigma(b)\tau_a^+};
        \,
        \tau_a^+<\tau_b^-\wedge\zeta
    \right]
    =
    \frac{W_V^{(-\sigma(b))}(x,b)}
         {W_V^{(-\sigma(b))}(a,b)}.
\end{align*}
Consequently,
\begin{align*}
    e^{\sigma(b)t}
    \E_x\left[
        W_V^{(-\sigma(b))}(V_t,b);
        \,
        t<\tau_a^+\wedge\tau_b^-\wedge\zeta
    \right]
    \leq
    W_V^{(-\sigma(b))}(x,b).
\end{align*}

By \cref{sigma-monotonicity},
\[
W_V^{(-\sigma(b))}(y,b)>0,
\qquad y\in(b,\infty).
\]
Therefore, letting \(a\to\infty\) and applying the monotone convergence
theorem, we obtain
\begin{align*}
    e^{\sigma(b)t}
    \E_x\left[
        W_V^{(-\sigma(b))}(V_t,b);
        \,
        t<\tau_b^-\wedge\zeta
    \right]
    \leq
    W_V^{(-\sigma(b))}(x,b).
\end{align*}
Thus, \(W_V^{(-\sigma(b))}(\cdot,b)\) is a
\(\sigma(b)\)-subinvariant function. Since the semigroup
\(Q\) is \(\sigma(b)\)-recurrent by \cref{lem:rho_recurrent},
Theorem~3 of \cite{TuominenTweedie1979} implies that
\(W_V^{(-\sigma(b))}(\cdot,b)\) is, in fact, a
\(\sigma(b)\)-invariant function for \(Q\).

\item[(ii)] We first show that
\[
\inf_{y\in B}W_V^{(-\sigma(b))}(y,b)>0
\]
for every bounded Borel set \(B\subset(b,\infty)\) satisfying
\(\inf B>b\).

Choose \(a>\sup B\). By \cref{two-sided-exit}, for every \(y\in B\),
\begin{align*}
    W_V^{(-\sigma(b))}(y,b)
    &=
    \E_y\left[
        e^{\sigma(b)\tau_a^+};
        \,
        \tau_a^+<\tau_b^-\wedge\zeta
    \right]
    W_V^{(-\sigma(b))}(a,b) \\
    &\geq
    \inf_{z\in B}
    \P_z\left(
        \tau_a^+<\tau_b^-\wedge\zeta
    \right)
    W_V^{(-\sigma(b))}(a,b).
\end{align*}
The right-hand side is strictly positive, as already established.
Moreover, \cref{unifrom-minorizing-measure}  implies that
for any \(\varepsilon>0\) and any sufficiently large $t$,
\begin{align*}
    \inf_{y\in B}
    \P_y\left(
        \tau_a^+<\tau_b^-\wedge\zeta
    \right)
    &\geq
    \inf_{y\in B}
    \P_y\left(
        V_t\in(a,a+\varepsilon);
        \,
        t<\tau_b^-
    \right)
    >0,
\end{align*}
where we have used that $V$ makes new maxima in a continuous way.
Consequently,
\begin{equation}\label{W-strict-set-positivity}
    \inf_{y\in B}W_V^{(-\sigma(b))}(y,b)>0.
\end{equation}
    
By \cref{lem:positive-recurrence}, the semigroup \(Q\) is
\(\sigma(b)\)-positive recurrent. Hence, \eqref{W-strict-set-positivity} along with Theorem~5(i) of
\cite{TuominenTweedie1979} implies that, for every \(x\in N^c\) and
every bounded Borel set \(B\subset(b,\infty)\) with \(\inf B>b\),
\begin{align}\label{tweedi-point-wise-thm}
    \lim_{s\to\infty}
    \frac{1}{s}
    \int_0^s
        e^{\sigma(b)t}
        Q_t\Ind_B(x)
    \,\mathrm{d}t
    =
    W_V^{(-\sigma(b))}(x,b)\nu^*(B),
\end{align}
for some measure \(\nu^*\), which is a \(\sigma(b)\)-subinvariant measure.

We now identify the measure $\nu^*$ explicitly. Fix $x\in(b,\infty)$ and
a bounded Borel set $B\subset(b,\infty)$ with $\inf B>b$. By
\cref{cor_1}\textnormal{(i)}, the resolvent representation in
\cref{lem:limit_potential} remains valid at
$q=-\sigma(b)+\varepsilon$ for any $\varepsilon>0$.
Hence,
\begin{align}
&\varepsilon\int_0^\infty
e^{(\sigma(b)-\varepsilon)t}Q_t\Ind_B(x)\,dt
\nonumber\\
&=
\frac{\varepsilon W_V^{(-\sigma(b)+\varepsilon)}(x,b)}
     {D^{(-\sigma(b)+\varepsilon)}(b)}
\left(
\int_{B\cap(0,\infty)}
e^{-\Phi_X(q_X-\sigma(b)+\varepsilon)y}\,dy
+
\int_{B\cap(b,0]}
D^{(-\sigma(b)+\varepsilon)}(y)\,dy
\right)
\nonumber\\
&\quad-
\varepsilon\int_{B\cap(0,\infty)}
W_X^{(q_X-\sigma(b)+\varepsilon)}(x-y)\,dy
-
\varepsilon\int_{B\cap(b,0]}
W_V^{(-\sigma(b)+\varepsilon)}(x,y)\,dy.
\label{eq:abelian-resolvent-expansion}
\end{align}

We look at the convergence on the right-hand side of \eqref{eq:abelian-resolvent-expansion},
dealing first with the term in parentheses and then with the following line.
Choose $\Gamma$ such that
\begin{equation*}
    \sigma(b)<\Gamma<q_X-\inf\Psi_X.
\end{equation*}
For all sufficiently small $\varepsilon>0$,
$-\sigma(b)+\varepsilon\in[-\Gamma,0)$. Therefore, by
\cref{lem:Dq_bound},
$D^{(-\sigma(b)+\varepsilon)}\rvert_{(b,0)}$ is bounded (uniformly in such $\varepsilon$)
by an integrable function.
% \begin{equation}\label{eq:D-dominating-function}
%     \left|D^{(-\sigma(b)+\varepsilon)}(y)\right|
%     \leq
%     G_Q(y)
%     :=
%     g(Q)W_Y^{(Q+q_Y)}(-y)
%     +
%     A\bigl(Q+q_Y,\Phi_X(\inf\Psi_X),y\bigr),
%     \qquad y\in[b,0].
% \end{equation}
% The function $G_Q$ is integrable on $[b,0]$.
% Indeed,
% $W_Y^{(Q+q_Y)}$ is continuous on $[0,-b]$, while the integrability of
% the second term follows \cref{lem:weighted-to-unweighted}.

We note that $\Phi_X(q_X-\sigma(b)+\epsilon) \downarrow \Phi_X(q_X-\sigma(b)) > 0$
as $\epsilon \downarrow 0$, and \cref{lem:Dq_analytic} gives us the pointwise
convergence $D^{(-\sigma(b)+\varepsilon)}(y) \to D^{(-\sigma(b))}(y)$.
Therefore, we can apply the dominated convergence theorem
to see that the term in parentheses in \eqref{eq:abelian-resolvent-expansion}
converges to $\nu(B)$.

Turning now to the second half of \eqref{eq:abelian-resolvent-expansion},
\cref{lem:Dq_bound} gives us a bound for $W_V^{(-\sigma(b)+\varepsilon)}(x,\cdot)$,
uniformly in $\varepsilon$, by an integrable function.
\cite[Lemma~4(iii)]{Bertoin1997} provides a similar bound for
$W_X^{(q_X-\sigma(b)+\varepsilon)}(x,\cdot)$. Therefore, the term under
consideration vanishes upon taking $\varepsilon$ to zero.

Finally, by \cref{lem:positive-recurrence}, $-\sigma(b)$ is a simple zero of
$q\mapsto D^{(q)}(b)$. Consequently,
\begin{equation*}
    \lim_{\varepsilon\downarrow0}
    \frac{\varepsilon}{D^{(-\sigma(b)+\varepsilon)}(b)}
    =
    c_D,
    \qquad
    c_D
    :=
    \left(
        \left.
        \frac{\partial}{\partial q}D^{(q)}(b)
        \right|_{q=-\sigma(b)}
    \right)^{-1}>0.
\end{equation*}
The analyticity of $q\mapsto W_V^{(q)}(x,b)$ also gives
\begin{equation*}
    W_V^{(-\sigma(b)+\varepsilon)}(x,b)
    \longrightarrow
    W_V^{(-\sigma(b))}(x,b).
\end{equation*}
Taking limits in
\eqref{eq:abelian-resolvent-expansion}, we obtain
\begin{equation}\label{eq:abelian-rank-one-limit}
    \lim_{\varepsilon\downarrow0}
    \varepsilon\int_0^\infty
    e^{(\sigma(b)-\varepsilon)t}Q_t\Ind_B(x)\,dt
    =
    c_DW_V^{(-\sigma(b))}(x,b)\nu(B).
\end{equation}

Hence, comparing \eqref{tweedi-point-wise-thm} and 
\eqref{eq:abelian-rank-one-limit} yields by an Abelian theorem
\begin{equation*}
    \nu^*(B)=c_D\nu(B).
\end{equation*}
This extends to arbitrary Borel sets $B$,
so $\nu^* = c_D \nu$.
It follows that $\nu$ is $\sigma(b)$-subinvariant, and consequently
$\sigma(b)$-invariant by
\cite[Theorem~4(iii)]{TuominenTweedie1979}
and \cref{lem:rho_recurrent}.
\item[(iii)] Once we have proved the previous parts, this is a direct consequence
of Theorem~3 in \cite{TuominenTweedie1979}.
\end{itemize}
\end{proof}

\section{Long-term behavior of the killed process}
\label{s:long-term}

We now want to study the long-term behaviour of $V$ killed upon first passage below zero.
For brevity, we define measures $\P^\dag_x$ under which we send $V$ to a cemetery 
state $\partial$ at the killing time $\tau_\partial \coloneqq \tau_b^-\wedge \zeta$,
so that $Q$ is the semigroup associated with $V$ under $(\P^\dag_x)_{x\in (b,\infty)}$.

Our primary goal is to show that the
$\sigma(b)$-invariant measure $\nu$ of $Q$ found in the previous section is the Yaglom limit
of $Q$, and that the convergence occurs at exponential rate.

Recall here the definition of the weighted total variation norm,
for a signed measure $\nu$ and a positive function $\psi$ on a space $E$:
\[
  \lVert \nu\rVert_{\mathrm{TV}(\psi)}
  = \sup\{ \lvert \nu f\rvert :  f\colon E \to \R\text{ measurable s.t. } \lvert f\rvert \le \psi\}.
\]
For $\psi=1$, we get the usual total variation norm, denoted
$\lVert \,\cdot\, \rVert_{\mathrm{TV}}$.

The measure $\nu$ in \cref{lem:rho_invariant_qsd} is finite, which can be seen from
the fact that $\Phi_X(-\sigma(b)+q_X)>0$ under \ref{a:R}, and \cref{lem:Dq_bound}.
This is sufficient to apply
Theorem~7 in \cite{TuominenTweedie1979}.
For Lebesgue-a.e.\ $x\in(b,\infty)$, we have the convergence
\begin{equation}\label{TV-a.s}
  \lim_{t\rightarrow \infty}
  e^{\sigma(b) t}\,\mathbb{P}_x\big(V_t\in\cdot;\; t<\tau_\partial\big)
  = c_D\nu(\cdot)\,W_{V}^{(-\sigma(b))}(x,b)
\end{equation}
in total variation.

We first strengthen this to hold for \emph{all} \(x \in (b, \infty)\).
\begin{theorem}\label{lem:quasi_stationary_all_x}
For every $x \in (b,\infty)$, we have
\[
\lim_{t \to \infty} 
\Big\|\, e^{\sigma(b) t}\,\mathbb{P}^\dag_x\big(V_t \in \cdot\;;\; t<\tau_\partial \big)
- c_D\nu(\cdot)\, W_V^{(-\sigma(b))}(x,b) \,\Big\|_{\mathrm{TV}} = 0.
\]
\end{theorem}

\begin{proof}
Fix $x\in(b,\infty)$. Choose $y\in(b,\infty)$ with $y>x$ such that \eqref{TV-a.s} holds at $y$ (by a simple application of Lebesgue density theorem). Decompose the measure
\[
\mu_{x,t}(\cdot):=e^{\sigma(b)t}\,\mathbb{P}^\dag_x\big(V_t\in\cdot;\; t<\tau_\partial \zeta\big)
\]
as
\[
\mu_{x,t}=\mu_{x,t}^{(1)}+\mu_{x,t}^{(2)},
\]
where
\[
\mu_{x,t}^{(1)}(\cdot):=e^{\sigma(b)t}\,\mathbb{P}^\dag_x\big(V_t\in\cdot;\; t<\tau_\partial \wedge\tau_y^+\big),
\]
and
\[
\mu_{x,t}^{(2)}(\cdot):=e^{\sigma(b)t}\,\mathbb{P}^\dag_x\big(V_t\in\cdot;\; \tau_y^+<t<\tau_\partial\big).
\]

By \cref{two-sided-exit}, along with a monotonicity of roots argument and \cref{sigma-monotonicity}, we have $\sigma(b, y)>\sigma(b)$ imply that there exists $\varepsilon>0$ with $\sigma(b, y)-\varepsilon>\sigma(b)$ and $\E_x^\dag\big[e^{(\sigma(b, y)-\varepsilon)(\tau_\partial \wedge\tau_y^+)}\big]<\infty.$ Therefore, 
\[
\mathbb{P}^\dag_x\big(t<\tau_\partial \wedge\tau_y^+\big)
=o\;\!\big(e^{-(\sigma(b, y)-\varepsilon)t}\big).
\]
Therefore the total variation norm of $\mu_{x,t}^{(1)}$ satisfies
\[
\|\mu_{x,t}^{(1)}\|_{\mathrm{TV}}
= e^{\sigma(b)t}\,\mathbb{P}_x\big(t<\tau_\partial\wedge\tau_y^+\big)
=o\;\!\big(e^{-(\sigma(y,b)-\varepsilon-\sigma(b))t}\big),
\]
as $t\to\infty$. Thus $\mu_{x,t}^{(1)}\to 0$ in total variation.

By the strong Markov property at time $\tau_y^+$ we may write for any measurable set $B$,
\[
\mu_{x,t}^{(2)}(B)
= \mathbb{E}_x\!\Big[ e^{\sigma(b)\tau_y^+}\,1_{\{\tau_y^+<\tau_\partial \wedge t\}}
\; \mu_{y,t-\tau_y^+}(B)\Big].
\]
We therefore obtain, using \cref{two-sided-exit} in the first line,
\begin{multline*}
  \big\| \mu_{x,t}^{(2)} - c_D W_V^{(-\sigma(b))}(x,b)\nu \big\|_{\mathrm{TV}}
  =
  \Bigl\lVert
    \mu_{x,t}^{(2)}
    - \mathbb{E}_x \big[ e^{\sigma(b)\tau_y^+}\Ind_{\{\tau_y^+<\tau_\partial\}}\big]
    c_D W_V^{(-\sigma(b))}(y,b) \nu
  \Bigr\lVert_{\mathrm{TV}}
  \\
  \le
  \mathbb{E}_x\!\Big[ e^{\sigma(b)\tau_y^+}\,1_{\{\tau_y^+<\tau_\partial \}}
\; \big\| \mu_{y,t-\tau_y^+} - c_D W_V^{(-\sigma(b))}(y,b)\nu \big\|_{\mathrm{TV}}\Big]
\end{multline*}
By assumption we have convergence of $\mu_{y,t}$ as $t\to\infty$.
To apply dominated convergence in the expectation, we again note that 
since $\sigma(b, y)>\sigma(b),$ we have $\mathbb{E}_x\!\Big[ e^{\sigma(b)\tau_y^+}\,1_{\{\tau_y^+<\tau_\partial \}}
\Big]<\infty$.
Therefore, we have
\[
\big\| \mu^{(2)}_{x,t} - c_D \nu(\cdot)\,W_V^{(-\sigma(b))}(x,b)\big\|_{\mathrm{TV}} \longrightarrow 0,
\qquad t\to\infty.
\]
Combining this with the vanishing of $\|\mu_{x,t}^{(1)}\|_{\mathrm{TV}}$, we obtain
\[
\big\| \mu_{x,t} - c_D \nu(\cdot)\,W_V^{(-\sigma(b))}(x,b)\big\|_{\mathrm{TV}} \longrightarrow 0,
\qquad t\to\infty.
\]
This completes the proof.
\end{proof}

We now show that this convergence occurs at exponential rate.
It can be shown, using
\cite[Theorem~1.1]{ChampagnatVillemonais2016QSD}, that this cannot be uniform
in $x$. However, we are able to quantify the dependence on $x$.
\begin{theorem}
\label{lem:combined-conditions}
Let $L=[d_1,d_2]$ for any $d_1 \in (b,0)$ and $d_2>0$. Denote the hitting time
of $L$ by
\[
\tau_L = \inf\{t > 0 : V_t \in L\}.
\]
The following conditions hold:
\begin{enumerate}
    \item 
    \textbf{Local Dobrushin coefficient.}
    There exist $t_1, c_1 > 0$ and a probability measure $\mu$ such that for all $x \in L$,
    \begin{align}\label{dub-2}     
    \mathbb{E}^\dag_x\big(V_{t_1} \in \cdot \; ; \; t_1 < \tau_\partial \big) 
    \ge c_1 \, \mu(\cdot \cap L).
    \end{align}
    \item 
    \textbf{Lyapunov condition.} 
    There exist 
    $\lambda_1>0$
    such that
    \[
      \psi_1(x) = \E_x^\dag\!\left(e^{\lambda_1 (\tau_L\wedge \tau_\partial)} \right) < \infty, 
      \qquad x \in (b,\infty),
    \]
    and $0<\lambda_2<\lambda_1$, $t_2,c_2>0$ such that
    \begin{align}
    \E_x^\dag\!\left(\psi_1(V_{t_2}) \, \mathbf{1}_{\{t_2 < \tau_L \wedge \tau_\partial\}}\right) 
    &\leq e^{-\lambda_1t_2} \, \psi_1(x), 
    &&x \in (b, \infty), \label{eq:lyap1}\\
    \E_x^\dag \!\left(\psi_1(V_t) \, \mathbf{1}_{\{t < \tau_\partial\}}\right) 
    &\leq c_2, 
    && x \in L, \, t \in [0,t_2], \label{eq:lyap2}\\
    \lim_{t\to\infty} e^{\lambda_2 t} \, \E_x^\dag(V_{t} \in L) 
    &= +\infty
    && x \in L. \label{eq:lyap3}
    \end{align}
    \item \textbf{Local Harnack inequality.} There exists $c_3 > 0$ such that
    \[
    \sup_{t \geq 0}
    \frac{\sup_{y \in L} \mathbb{P}_y^\dag(t < \tau_\partial)}
    {\inf_{y \in L} \mathbb{P}_y^\dag(t < \tau_\partial)}
    \le c_3.
    \]
\end{enumerate}
Consequently, there exist
$\alpha \in (0, 1)$
and $C>0$ such that for all probability measures $\mu'$ on $(b, \infty)$ satisfying $\mu'(\psi_1)<\infty$ and $\mu'(\psi_2)>0,$ 
\[
\Big\|
e^{\sigma(b) t}
\mathbb{P}_{\mu'}^\dag\!\big( V_t \in \,\cdot \mathbin{\big|} t < \tau_\partial \big)
- \frac{\nu(\cdot)}{\nu(1)}
\Big\|_{\mathrm{TV}(\psi_1)}
\leq C \alpha^{t} \frac{\mu'(\psi_1)}{\mu'(\psi_2)},
\qquad t>0,
\]
where $\psi_2(x)=\sum_{k=0}^{n_0} e^{\lambda_2 kt_2} \P^\dag_x(V_{kt_2}\in L)$,
for some $n_0$ large enough.
\end{theorem}
\begin{proof}
\textbf{(i) Local Dobrushin coefficient.}
This is a corollary of \cref{unifrom-minorizing-measure}, by choosing $A=L$.

\medskip
\textbf{(ii) Lyapunov condition.}
For $x > b$ and $\lambda > 0$, we decompose:
\begin{align}
\E_x^\dag\!\left(e^{\lambda (\tau_L \wedge \tau_\partial)}\right)
&= \E_x\!\left(e^{\lambda (\tau_L\wedge \zeta)}; \,  \tau_L\wedge \zeta < \tau_0^-\right)
 + \E_x^\dag\!\left(e^{\lambda (\tau_L \wedge \tau_\partial)}; \, \tau_0^- \le \tau_L\wedge \zeta\right).
 \label{eq:lyap-decompose}
\end{align}
For the second term, by the strong Markov property at $\tau_0^-$,
\begin{align}
\E_x^\dag\!\left(e^{\lambda (\tau_L \wedge \tau_\partial)}; \, \tau_0^- \le \tau_L\wedge \zeta \right) 
&= \E_x\!\Big( e^{\lambda \tau_0^-} 
\, \E^\dag_{V_{\tau_0^-}}\!\left(e^{\lambda (\tau_L \wedge \tau_\partial)}\right); \, \tau_0^- \le \tau_L\wedge \zeta\Big) \nonumber \\
&\leq \E_x\!\left(e^{\lambda \tau_0^-}; \, \tau_0^- \le \tau_L\wedge \zeta\right) \, 
\sup_{y \in (b,0)} \E^\dag_y\!\left(e^{\lambda (\tau_L \wedge \tau_\partial)}\right).
\label{eq:lyap-inter2}
\end{align}
For $d_1\le y<0$, $\P_y(\tau_L=0) = 1$.
For $b<y<d_1$, we have for $\delta>0$ that
\begin{equation}
  \label{eq:lyap-decompose-2}
  \E^\dag_y\!\left(e^{(\sigma(b, d_1)-\delta)(\tau_L \wedge \tau_\partial)}\right)
  = \E_y\!\left(e^{(\sigma(b,d_1)-\delta)\tau_{d_1^+} \wedge \tau_b^- \wedge \zeta}\right),
\end{equation}
and by \cref{exist_resolvent} and \cref{cor_1}, when $\delta$ is sufficiently small,
this can be expressed in terms of the resolvent density
$r_{(b,d)}^{(\sigma(b,d_1)-\delta)}(y,z)$, for $b<y,z<d_1$.
We now proceed to show that this is bounded.

By definition, $W_V^{(-\sigma(b, d_1)+\delta)}(d_1,b) > 0$ and hence,  
we have $$|W_V^{(-\sigma(b, d)+\delta)}(y,b)|\leq W_V^{(\sigma(b, d)-\delta)}(y,b) W_V^{(\sigma(b, d)-\delta)}(d,b)<\infty$$
(using \cref{lem:Dq_bound} for the first inequality and \cref{two-sided-exit} for the second).
This is enough to show that $r_{(b,d)}^{(\sigma(b,d)-\delta)}$, and hence,
\eqref{eq:lyap-decompose-2} is bounded in $y$
for any $0 < \lambda \le \sigma(b,d)-\delta$.

Now, for any $a>0$, the process $V$ up to time $\tau_a^-\wedge \zeta$ is just the Lévy process
$X$ with additional killing at rate $q_X$. Using
Corollary~8.8 in \cite{Kyprianou2014}, we have
\begin{equation}
\label{eq:levy-moment}
\E_x\!\left(e^{(q_X+|\inf \psi_X|) \, (\tau_a^- \wedge \zeta)} ; \tau_a^-<\infty \right) < \infty.
\end{equation}
This gives us in particular that \eqref{eq:lyap-inter2} is finite when $\lambda$ is taken
equal to
\[
\lambda_1 \coloneqq \min\big\{ q_X+|\inf \psi_X|, \, \sigma(b, d_1) - \delta \big\}.
\]
It is left to us to show the finiteness of the first term in \eqref{eq:lyap-decompose}.
Recall that $L = [d_1,d_2]$ and that $d_2>0$. So
\begin{align*}
\E_x\!\left(e^{|\inf \psi_X| \, \tau_L \wedge \zeta}\; ; \tau_L<\tau_0^-\right)
&=\E_x\!\left(e^{|\inf \psi_X| \, \tau_{d_2}^-\wedge\zeta}\; ; \tau_{d_2}^-<\tau_0^-\right)
\\
&\leq \E_x\!\left(e^{|\inf \psi_X| \, \tau_{d_2}^-\wedge\zeta}\; ; \tau_{d_2}^-<\infty\right)<\infty,  
\end{align*}
using again \eqref{eq:levy-moment}.
It follows that the first term in \eqref{eq:lyap-decompose} is also finite
at $\lambda=\lambda_1$. Putting together the pieces, we have that
\[
\psi_1(x) = \E_x^\dag\!\left(e^{\lambda_1 (\tau_L \wedge \tau_\partial)}\right) < \infty,
\qquad x \in (b,\infty).
\]
According to Lemma~3.6 in \cite{ChampagnatVillemonais2023GeneralCriteria}, this
establishes \eqref{eq:lyap1} and \eqref{eq:lyap2}.

Recall from Theorem~\ref{lem:quasi_stationary_all_x} that
\[
e^{\sigma(b) t}\, \P_x^\dag(V_{t}\in L; \, t< \tau_\partial) \;\longrightarrow\; C\nu(L)\,W_{V}^{(-\sigma(b))}(x,b).
\]
Thus, we may choose $\lambda_2 = \sigma(b) + \delta$ for any $\delta>0$,
so that $e^{\lambda_2 t}\P_x^\dag(V_t \in L;\; t < \tau_\partial) \to \infty.$  

If we take $\delta$ sufficiently small, then \cref{sigma-monotonicity} ensures 
\[
\sigma(b) + \delta < |\inf \psi_X| +q_X
\quad\text{and}\quad 
\sigma(b) + \delta<\sigma(b, d) - \delta .
\]
Hence $\lambda_2 < \lambda_1$, as required. 
\medskip

\item \textbf{(iii) Local Harnack inequality.}   
We take motivation from \cite{VillemonaisWatson2025GrowthFragmentationQSD}.
It was shown in \eqref{inf-lower-bound-clean} that there exists $t_L > 0$ such that
\[
\inf_{x \in L}\, \mathbb{P}^\dag_x(H_{d_2} < t_L\wedge \tau_\partial) > 0,
\]
We now strengthen this to  
\begin{equation}\label{eqn:minorization-harnack-use}
    \inf_{x, y \in L}\, \mathbb{P}^\dag_x(H_{y} < 3t_L\wedge\tau_\partial) > 0.
\end{equation}
Indeed, by \cref{irreducibility}, we have (possibly making $t_L$ larger) that
\begin{align*}
     \mathbb{P}^\dag_{d_2}(H_{d_1} < t_L\wedge \tau_\partial) > 0.
\end{align*} 
Therefore, for $x,y\in L$,
\begin{align*}
    \mathbb{P}^\dag_x(H_{y} < 3t_L\wedge \tau_\partial)
    &\ge \mathbb{P}^\dag_x(H_{d_2} < t_L\wedge \tau_\partial)
    \mathbb{P}^\dag_{d_2}(H_{d_1} < t_L\wedge \tau_\partial)
    \mathbb{P}^\dag_{d_1}(H_y < t_L\wedge \tau_\partial)\\
    &\ge \mathbb{P}^\dag_x(H_{d_2} < t_L\wedge \tau_\partial)
    \mathbb{P}^\dag_{d_2}(H_{d_1} < t_L\wedge \tau_\partial)
    \mathbb{P}^\dag_{d_1}(H_{d_2} < t_L\wedge \tau_\partial).
\end{align*}
% where the last line follows because
% \begin{align*}
%     \mathbb{P}_{d_1}(H_{d_2} < t_L\wedge \zeta \wedge \tau_b^-)& = \P_{d_1}(H_y<t_L\wedge \zeta \wedge \tau_b^-\; ;\P_y(H_{d_2}<(t_L-k)\wedge \zeta \wedge \tau_b^-)|_{H_y=k})\\
%     &\leq \P_{d_1}(H_y<t_L\wedge \zeta \wedge \tau_b^-).
% \end{align*}
The right-hand side is bounded below in $x$, which gives us \eqref{eqn:minorization-harnack-use}.
Next, for any \(t > 3t_L\), we have
\[
\mathbb{P}^\dag_x(t < \tau_\partial)
  \ge
  \mathbb{E}^\dag_x\!\left( \mathbf{1}_{\{H_y < 3t_L\wedge \tau_\partial\}}
    \, \mathbb{P}^\dag_y(t - 3t_L < \tau_\partial) \right)
  \;\ge\;
  \mathbb{P}^\dag_x(H_y < 3t_L\wedge \tau_\partial) \, \mathbb{P}^\dag_y(t < \tau_\partial).
\]
This implies
\[
  \frac{\mathbb{P}^\dag_y(t < \tau_\partial)}{\mathbb{P}^\dag_x(t < \tau_\partial)}
  \le
  \frac{1}{\mathbb{P}^\dag_x(H_y < 3t_L\wedge \tau_\partial)},
\]
which with \eqref{eqn:minorization-harnack-use} is the required bound for \(t > 3t_L\).

For \(t < 3t_L\), we first choose $I=(d_2-\varepsilon, d_2+\varepsilon)$ such that $d_2-\varepsilon>0$ and define $\tau_{I^c} = \inf\{t\geq 0: V_t \notin I\}.$ Now, we observe that for all $x \in L,$
\[
\mathbb{P}^\dag_x(t < \tau_\partial)
  \ge
  \mathbb{P}^\dag_x(3t_L < \tau_\partial)
  \ge
 \mathbb{P}^\dag_x(H_{d_2}<3t_L \wedge \tau_\partial) \P_{d_2}(\tau_{I^c}>3t_L\wedge \zeta) 
\]
It follows from \eqref{eqn:minorization-harnack-use} and \cref{lem:positivity_survival_general} that the right-hand side is bounded below by a positive number.
Thus
\[
  \frac{\mathbb{P}^\dag_y(t < \tau_\partial)}{\mathbb{P}^\dag_x(t < \tau_\partial)}
  \le
  \frac{1}{\inf_{x \in L} \mathbb{P}^\dag_x(H_{d_2}<3t_L \wedge \tau_\partial) \P_{d_2}(\tau_I>3t_L\wedge \zeta)} < \infty,
\]
and this completes the proof of (iii).

The final conclusion is obtained by Theorem~3.5 in \cite{ChampagnatVillemonais2023GeneralCriteria},
noting the reparametrisation $\gamma_1 = e^{-\lambda_1}$ and $\gamma_2 = e^{-\lambda_2}$.
In addition to the points above, the theorem also demands that
$\tau_L$ is a strong Markov time, in the sense that for all bounded measurable $f$ and $t \geq 0$,
    \begin{align*}
    \mathbb{E}^\dag_x \!\left[ f(V_t)\,\mathbf{1}_{\{\tau_L \leq t < \tau_\partial\}} \right] 
    &= \mathbb{E}^\dag_x \!\left[ \mathbf{1}_{\{\tau_L \leq t \wedge \tau_\partial\}} 
    \,\mathbb{E}^\dag_{V_{\tau_L}} \!\left[ f(V_{t-u})\,\mathbf{1}_{\{t-u < \tau_\partial\}} \right]_{u=\tau_L} \right]. 
    \end{align*}
However, since $L$ is closed, this is a straightforward consequence of the strong Markov
property of $V$.
\end{proof}

\appendix

\section{Properties of the scale function}

\label{appendix:scale_function}
\label{appendix:proof-lemma-excursion}

\begin{proposition}\label{prop:WV-analytic}
For every $b<0$ and $x>b$, the map
\begin{equation*}
    q\longmapsto W_V^{(q)}(x,b)
\end{equation*}
admits an entire extension to $\C$. More precisely,
\begin{equation}\label{eq:WV-power-series}
    W_V^{(q)}(x,b)
    =
    \sum_{n=0}^{\infty}
    q^n W_V^{*(n+1)}(x,b),
    \qquad q\in\C.
\end{equation}
\end{proposition}

\begin{proof}
Fix $b<x<a$, where $a>\max\{x,0\}$, and let
\begin{equation*}
    T_{b,a}:=\tau_a^+\wedge\tau_b^-.
\end{equation*}
Let $q,r\geq0$ with $q\neq r$ and consider
\begin{equation*}
    I_{q,r}(x)
    :=
    \E_x\left[
        \int_0^{T_{b,a}}
        e^{-rt}W_V^{(q)}(V_t,b)\,\mathrm{d}t
    \right].
\end{equation*}

We calculate this quantity in two ways. By the strong Markov property at
time $t$ and \cref{two-sided-exit},
\begin{align*}
&\E_x\left[
    e^{-q\tau_a^+};
    \ t<\tau_a^+<\tau_b^-\wedge\zeta
\right]\\
&\qquad=
\E_x\left[
    \Ind_{\{t<T_{b,a}\}}e^{-qt}
    \E_{V_t}\left(
        e^{-q\tau_a^+};
        \ \tau_a^+<\tau_b^-\wedge\zeta
    \right)
\right]\\
&\qquad=
\frac{e^{-qt}}{W_V^{(q)}(a,b)}
\E_x\left[
    W_V^{(q)}(V_t,b);
    \ t<T_{b,a}
\right].
\end{align*}
Consequently, 
\begin{align}
I_{q,r}(x)
&=
W_V^{(q)}(a,b)
\E_x\left[
    e^{-q\tau_a^+}
    \int_0^{\tau_a^+}e^{-(r-q)t}\,\mathrm{d}t;
    \ \tau_a^+<\tau_b^-\wedge\zeta
\right]
\nonumber\\
&=
\frac{W_V^{(q)}(a,b)}{r-q}
\E_x\left[
    e^{-q\tau_a^+}-e^{-r\tau_a^+};
    \ \tau_a^+<\tau_b^-\wedge\zeta
\right]
\nonumber\\
&=
\frac{1}{r-q}
\left(
    W_V^{(q)}(x,b)
    -
    \frac{W_V^{(q)}(a,b)}
         {W_V^{(r)}(a,b)}
    W_V^{(r)}(x,b)
\right).
\label{eq:Iqr-exit}
\end{align}

On the other hand, applying \cref{exist_resolvent} with the function
$z\mapsto W_V^{(q)}(z,b)$, we obtain
\begin{align}
I_{q,r}(x)
&=
\int_b^a
W_V^{(q)}(z,b)r_{(b,a)}^{(r)}(x,z)\,\mathrm{d}z
\nonumber\\
&=
\frac{W_V^{(r)}(x,b)}
     {W_V^{(r)}(a,b)}
\int_b^a
W_V^{(r)}(a,z)W_V^{(q)}(z,b)\,\mathrm{d}z
-
\int_b^x
W_V^{(r)}(x,z)W_V^{(q)}(z,b)\,\mathrm{d}z
\nonumber\\
&=
\frac{W_V^{(r)}(x,b)}
     {W_V^{(r)}(a,b)}
\bigl(W_V^{(r)}*W_V^{(q)}\bigr)(a,b)
-
\bigl(W_V^{(r)}*W_V^{(q)}\bigr)(x,b).
\label{eq:Iqr-resolvent}
\end{align}

Equating \eqref{eq:Iqr-exit} and
\eqref{eq:Iqr-resolvent} and rearranging gives
\begin{align}
&\frac{
W_V^{(q)}(x,b)
+
(r-q)\bigl(W_V^{(r)}*W_V^{(q)}\bigr)(x,b)}
{W_V^{(r)}(x,b)}
\nonumber\\
&\qquad=
\frac{
W_V^{(q)}(a,b)
+
(r-q)\bigl(W_V^{(r)}*W_V^{(q)}\bigr)(a,b)}
{W_V^{(r)}(a,b)}.
\label{eq:WV-ratio-independent-x}
\end{align}
Observe that, the expression on the left hand side is independent of
$x\in(b,a)$.

We now let $x\downarrow b$. Since $b<0$, for all $x$ sufficiently close
to $b$, we have
\begin{equation*}
    W_V^{(q)}(x,b)
    =
    W_Y^{(q+q_Y)}(x-b).
\end{equation*}
Next, we first note that
\begin{equation}\label{eq:classical-near-zero-ratio}
    \lim_{x\downarrow b}
    \frac{W_V^{(q)}(x,b)}
         {W_V^{(r)}(x,b)}
    =1.
\end{equation}
Indeed, writing $\delta=x-b$ and assuming, without loss of generality,
that $q\geq r$, we have
\begin{align*}
0
&\leq
1-
\frac{W_Y^{(r+q_Y)}(\delta)}
     {W_Y^{(q+q_Y)}(\delta)}\\
&=
(q-r)
\frac{
\int_0^\delta
W_Y^{(q+q_Y)}(\delta-u)
W_Y^{(r+q_Y)}(u)\,\mathrm{d}u}
{W_Y^{(q+q_Y)}(\delta)}
\leq
(q-r)\int_0^\delta
W_Y^{(r+q_Y)}(u)\,\mathrm{d}u,
\end{align*}
and this tends to zero as $\delta \rightarrow 0$.

Moreover, for $b<x<0$,
\begin{align*}
&\frac{
\bigl(W_V^{(r)}*W_V^{(q)}\bigr)(x,b)}
{W_V^{(r)}(x,b)}
\nonumber\\
&=
\frac{
\int_b^x
W_Y^{(r+q_Y)}(x-z)
W_Y^{(q+q_Y)}(z-b)\,\mathrm{d}z}
{W_Y^{(r+q_Y)}(x-b)}
\nonumber\\
&\leq
\int_b^x
W_Y^{(q+q_Y)}(z-b)\,\mathrm{d}z,
\end{align*}
using the monotonicity of
$W_Y^{(r+q_Y)}$. It is clear that the right-hand side tends to zero as
$x\downarrow b$. It follows from
\eqref{eq:WV-ratio-independent-x} that 
\begin{equation}\label{eq:WV-resolvent-positive}
    W_V^{(q)}(x,b)-W_V^{(r)}(x,b)
    =
    (q-r)
    \bigl(W_V^{(r)}*W_V^{(q)}\bigr)(x,b),
    \qquad q,r\geq0.
\end{equation}
Taking $r=0$ in \eqref{eq:WV-resolvent-positive}, we obtain the Volterra
equation
\begin{equation}\label{eq:WV-Volterra}
    W_V^{(q)}(x,b)
    =
    W_V(x,b)
    +
    q\int_b^x
    W_V(x,z)W_V^{(q)}(z,b)\,\mathrm{d}z,
    \qquad q\geq0.
\end{equation}
We next solve \eqref{eq:WV-Volterra} by iteration. 
Iterating \eqref{eq:WV-Volterra}, for every $n\geq1$, gives
\begin{align}
W_V^{(q)}(x,b)
&=
\sum_{k=0}^{n-1}
q^k W_V^{*(k+1)}(x,b)
+
q^n
\int_b^x
W_V^{*n}(x,z)W_V^{(q)}(z,b)\,\mathrm{d}z.
\label{eq:WV-Volterra-iteration}
\end{align}
Indeed, the assertion follows by induction from
\begin{equation*}
    W_V^{*(k+1)}(x,b)
    =
    \int_b^x
    W_V^{*k}(x,z)W_V(z,b)\,\mathrm{d}z.
\end{equation*}
It remains to show that the final term in
\eqref{eq:WV-Volterra-iteration} converges to zero.
Recall that
\begin{equation*}
    F_x(z)
    :=
    \int_z^x W_V(x,u)\,\mathrm{d}u.
\end{equation*}
The above quantity is well defined by the same argument in \cref{lem:rho_invariant_qsd}. We claim that
\begin{equation}\label{eq:WV-factorial-bound}
    W_V^{*n}(x,y)
    \leq
    \frac{W_V(x,y)}{(n-1)!}
    F_x(y)^{n-1},
    \qquad n\geq1.
\end{equation}
The assertion is immediate for $n=1$. Using \cref{two-sided-exit}, we have
\begin{equation*}
    W_V(z,y)\leq W_V(x,y),
    \qquad y\leq z\leq x.
\end{equation*}
Suppose it holds for some
$n\geq1$. Then,
\begin{align*}
W_V^{*(n+1)}(x,y)
&=
\int_y^x
W_V^{*n}(x,z)W_V(z,y)\,\mathrm{d}z\\
&\leq
\frac{W_V(x,y)}{(n-1)!}
\int_y^x
W_V(x,z)F_x(z)^{n-1}\,\mathrm{d}z.
\end{align*}
Since
\begin{equation*}
    F_x(z)=\int_z^xW_V(x,u)\,\mathrm{d}u
    \quad\text{and}\quad
    \mathrm{d}F_x(z)
    =
    -W_V(x,z)\,\mathrm{d}z,
\end{equation*}
we have
\begin{equation*}
\int_y^x
W_V(x,z)F_x(z)^{n-1}\,\mathrm{d}z
=
\frac{F_x(y)^n}{n}.
\end{equation*}
This proves \eqref{eq:WV-factorial-bound} by induction. Consequently,
\begin{align}
0
&\leq
q^n
\int_b^x
W_V^{*n}(x,z)W_V^{(q)}(z,b)\,\mathrm{d}z
\nonumber\\
&\leq
W_V^{(q)}(x,b)
\frac{\bigl(qF_x(b)\bigr)^n}{n!}
\longrightarrow0,
\qquad n\to\infty.
\label{eq:WV-Volterra-remainder}
\end{align}
Letting $n\to\infty$ in
\eqref{eq:WV-Volterra-iteration}, we conclude that
\begin{equation*}
    W_V^{(q)}(x,b)
    =
    \sum_{k=0}^{\infty}
    q^k W_V^{*(k+1)}(x,b),
    \qquad q\geq0.
\end{equation*}
Thus, for every fixed $x>b$, the series in
\eqref{eq:WV-power-series} converges locally uniformly in $q\in\C$ and
defines an entire extension of
$q\mapsto W_V^{(q)}(x,b)$.
\end{proof}
\paragraph{Acknowledgement}
Ramkrishna Jyoti Samanta acknowledges support from the Additional Funding Programme for Mathematical Sciences, delivered by EPSRC (EP/V521917/1) and the Heilbronn Institute for Mathematical Research.
\bibliography{reference}
\end{document}